\documentclass{amsart}

\usepackage[utf8]{inputenc}
\usepackage[pdfencoding=unicode, pageanchor=true]{hyperref}

\usepackage{xspace}
\usepackage{marvosym}

\usepackage{amsfonts,amsmath,amssymb,amsthm}
\usepackage{url}
\usepackage{enumitem}

\newtheorem{theorem}{Theorem}[section]
\newtheorem{lemma}[theorem]{Lemma}
\newtheorem{proposition}[theorem]{Proposition}
\newtheorem{corollary}[theorem]{Corollary} 
\newtheorem{conjecture}[theorem]{Conjecture}

\theoremstyle{definition}
\newtheorem{definition}[theorem]{Definition}

\theoremstyle{remark}
\newtheorem{remark}[theorem]{Remark}

\newtheorem{question}[theorem]{Question}
\newtheorem{example}[theorem]{Example}

\usepackage{algorithmicx}
\usepackage{algorithm}
\usepackage{algpseudocode}

\algdef{SE}[DOWHILE]{Do}{DoWhile}{\algorithmicdo}[1]{\algorithmicwhile\ #1}
\algrenewcommand{\algorithmiccomment}[1]{\hfill $\rhd$ \emph{#1}}
\algrenewcommand{\algorithmicrequire}{\textbf{Input:}}
\algrenewcommand{\algorithmicensure}{\textbf{Output:}}
\algnewcommand{\Or}{\textbf{or}}
\algnewcommand{\And}{\textbf{and}}
\algnewcommand{\Not}{\textbf{not}\,}

\newcommand\cG{{\mathcal G}}

\newcommand\cM{{\mathcal M}}
\newcommand\cN{{\mathcal N}}

\newcommand\cP{{\mathcal P}}

\newcommand\cS{{\mathcal S}}
\newcommand\cT{{\mathcal T}}

\newcommand\CC{{\mathbb C}}

\newcommand\RR{{\mathbb R}}
\newcommand\TT{{\mathbb T}}

\newcommand\ZZ{{\mathbb Z}}

\newcommand\SetOf[2]{\left\{\left.#1\vphantom{#2}\ \right|\ #2\vphantom{#1}\right\}}

\newcommand\Dprod[2]{\Delta_{#1} \times \Delta_{#2}}

\DeclareMathOperator*{\conv}{\operatorname{conv}}

\newcommand{\lno}{\ell}
\newcommand{\rno}{r}
\newcommand{\lnoset}{L}
\newcommand{\rnoset}{R}
\newcommand{\unit}[1]{e_{#1}}
\newcommand{\simplat}[2]{\Delta^{\ZZ}(#2,#1)}

\usepackage{tikz,tikz-cd}
\tikzcdset{}

\usepackage{mytikz}

\usepackage[foot]{amsaddr}
\author{Georg Loho}
\address[Georg Loho]{London School of Economics, London, United Kingdom}
\email{g.loho@lse.ac.uk}
\author{Ben Smith}
\address[Ben Smith]{University of Manchester and Heilbronn Institute for Mathematical Research, United Kingdom}
\email{benjamin.smith-3@manchester.ac.uk}
\thanks{Georg Loho was supported by the Swiss National Science Foundation (SNSF) within the project \emph{Convexity, geometry of numbers, and the complexity of integer programming (Nr.~163071)} and profited from the fruitful atmosphere at the Institut Mittag-Leffler within the Research Program `Tropical Geometry, Amoebas, and Polytopes'. Ben Smith was supported by the EPSRC grant \emph{Arrangements of tropical linear spaces (1673882)}. } 

\keywords{Matching field, triangulation, lattice points, product of simplices, linkage property, bipartite graph}

\begin{document}

\title{Matching fields and lattice points of simplices}

\begin{abstract}
  We show that the Chow covectors of a linkage matching field define a bijection between certain degree vectors and lattice points, and we demonstrate how one can recover the linkage matching field from this bijection.
  This resolves two open questions from Sturmfels \& Zelevinsky (1993) on linkage matching fields.
  For this, we give an explicit construction that associates a bipartite incidence graph of an ordered partition of a common set to each lattice point in a dilated simplex.

  Given a triangulation of a product of two simplices encoded by a set of spanning trees on a bipartite node set, we similarly prove that the bijection from left to right degree vectors of the trees is enough to recover the triangulation.
  As additional results, we show a cryptomorphic description of linkage matching fields and characterise the flip graph of a linkage matching field in terms of its prodsimplicial flag complex.
  Finally, we relate our findings to transversal matroids through the tropical Stiefel map.
\end{abstract}

\maketitle

\section{Introduction}

A \emph{matching field} is a set of perfect matchings on bipartite node sets $\sigma \sqcup [d]$, one for each $d$-subset $\sigma$ of an $n$-set $\lnoset$.
A natural example is the set of weight minimal matchings of size $d$ in a complete bipartite graph $K_{n,d}$ with generic edge weights.
Matching fields arising in this way are \emph{coherent}.
Without any further requirements, the set of matchings can be arbitrarily unstructured.
Our main objects of study will be \emph{linkage matching fields}.
They fulfil the additional property that each subset of matchings defined on a $(d+1)$-subset of $\lnoset$ is coherent.
Sturmfels \& Zelevinsky introduced  matching fields in~\cite{SturmfelsZelevinsky:1993} to study the Newton polytope of the product of all maximal minors of an $(n \times d)$-matrix of indeterminates $X = (x_{ji})$.
The linkage property occurs as a combinatorial description of the determinantal identity~\cite[Equation 0.1]{SturmfelsZelevinsky:1993}
\[
\sum_{j \in \tau} (-1)^{j}x_{ji}\, X_{\tau \setminus \{j\}} = 0 \mbox{ for all } i \in [d], \tau \in \binom{[n]}{d+1} \enspace , 
\]
where $X_{\sigma}$ is the minor of the rows labelled by a $d$-subset $\sigma \subseteq [n]$.
This is analogous to the motivation of the exchange axiom of a matroid from the Pl{\"u}cker relations.

\subsection{Motivation and Former Work} Linkage matching fields have already proven to be useful in several contexts.
The combination of the results in~\cite{SturmfelsZelevinsky:1993, BernsteinZelevinsky:1993} showed that the maximal minors of an $(n \times d)$-matrix of indeterminates form a universal Gr\"{o}bner basis of the ideal generated by them.
They occur in tropical linear algebra, as tropical determinants are just minimal matchings in a weighted bipartite graph, yielding a matching field in the generic case.
This was used in~\cite{RichterGebertSturmfelsTheobald} to devise a tropical Cramer's rule. 
Further on, avoiding the genericity assumption, a generalisation called `matching multifields' was employed to examine the structure of the image of the tropical Stiefel map in~\cite{FinkRincon:2015}.
Another recent work related to Grassmannians~\cite{MohammadiShaw:2018} uses matching fields to find toric degenerations.

\smallskip

	A collection of graphs associated to a matching field was introduced in~\cite{SturmfelsZelevinsky:1993}, which we refer to as the \emph{Chow covectors} of a matching field.
	They have a combinatorial characterisation as the minimal transversals to a linkage matching field, as shown in~\cite{BernsteinZelevinsky:1993}, and are a key combinatorial tool in the proof of the universal Gr\"{o}bner basis result mentioned earlier.
	Their initial introduction was as `brackets' to study the variety of degenerate matrices in $\CC^{n \times d}$.
	In particular, they give insight into the Chow form, a polynomial invariant that determines the variety.
	Sturmfels \& Zelevinksy showed that `extremal' terms of the Chow form can be recovered from the Chow covectors by taking their product as brackets.

\smallskip

In~\cite{CeballosPadrolSarmiento:2015}, the characterisation of the extendibility of a partial triangulation of a product of two simplices was built on a representation in terms of unions of linkage matching fields from~\cite{OhYoo-ME:2013}.
The latter can also be considered as a cryptomorphism for generic tropical oriented matroids.
After the introduction of tropical oriented matroids in~\cite{ArdilaDevelin:2009}, it was shown in~\cite{Horn1} that they are equivalent to the subdivisions of a product of two simplices.
For tropical oriented matroids which come from tropical point configurations, this cryptomorphism was already established earlier \cite{BabsonBillera:1998, JoswigLoho:2016}. 
As linkage matching fields are the building blocks of tropical oriented matroids, we propose to consider them as another matroid-like structure for tropical geometry.

Only recently, the first author proposed a variation of tropical oriented matroids, namely `signed tropical matroids', to develop an abstraction of tropical linear programming~\cite{Loho17} analogous to oriented matroid programming.
Note that the algorithm in the latter paper relies on the interplay of the linkage covectors, see Definition~\ref{def:linkage+covector}, of some matching fields derived from a triangulation of a product of two simplices. 

\subsection{Our Results}
The main tool for our considerations are \emph{topes}, which occur in the context of tropical oriented matroids \cite{ArdilaDevelin:2009}.
We generalise the concept of \emph{matching fields} to \emph{tope fields}.
While matching fields comprise a set of matchings, tope fields can be seen as sets of ordered partitions of a varying ground set.
We transfer the crucial \emph{linkage property} from matching fields to tope fields derived from a linkage matching field. 
This allows us to associate \emph{maximal topes} arising from a linkage matching field to the lattice points in a dilated simplex. 

\smallskip

We obtain an explicit construction of the Chow covectors.
Our approach leads to a representation derived from the \emph{maximal topes} of a linkage matching field in Proposition~\ref{prop:construct-Chow}.
This yields Theorem~\ref{thm:bijection-chow-covector-lattice-points}, resolving Conjecture~6.10 from~\cite{SturmfelsZelevinsky:1993} which was only resolved for coherent matching fields in~\cite{BernsteinZelevinsky:1993}.
Each bipartite graph induces a pair of lattice points, namely the pair of its left and right degree vector.
The theorem shows that the set of degree vector pairs for all Chow covectors gives rise to a bijection from $(n-d+1)$-subsets of $[n]$ to lattice points in the dilated simplex $(n-d+1)\Delta_{d-1}$.
Naturally, one can now ask if this bijection uniquely defines the matching field. 
This question is a generalisation of \cite[{Conjecture~6.8~b)}]{SturmfelsZelevinsky:1993} for linkage matching fields.
We answer it positively in Theorem~\ref{thm:lattice-points-Chow}.
These results allow us to give a cryptomorphic description of linkage matching fields in the form of \emph{tope arrangements}.

\smallskip

A similar claim for triangulations of $\Dprod{n-1}{d-1}$ was made after~\cite[Theorem~12.9]{Postnikov:2009}.
Describing a triangulation as a collection of trees in the sense of \cite[Proposition~7.2]{ArdilaBilley}, one also obtains a set of pairs of lattice points.
For a triangulation of $\Delta_{n-1} \times \Delta_{d-1}$, this yields a subset of $\simplat{d-1}{n} \times \simplat{n-1}{d}$, which denotes the product of the integer lattice points in the dilated simplices $(d-1)\Delta_{n-1}$ and $(n-1)\Delta_{d-1}$.
With essentially the same reasoning as in Theorem~\ref{thm:lattice-points-Chow}, we provide an explicit construction of a triangulation from these lattice point pairs.
At the same time, this result was proven also for more general root polytopes in~\cite{GalashinNenashevPostnikov2018}.
A comparison of their advances based on \emph{trianguloids} is given in Section \ref{sec:cryptomorphism}.

\smallskip

As an additional result, we show how the topes of a linkage matching field are encoded in its flip graph.
This leads to a characterisation of its prodsimplicial flag complex (see~\cite[Section 9.2.1]{Kozlov:2008}) as the `complex of topes' in Theorem~\ref{thm:prodsimplicial+complex+flip+graph}.

Furthermore, we initiate the study of a \emph{combinatorial Stiefel map}, generalising the construction in~\cite{FinkRincon:2015} and~\cite{HerrmannJoswigSpeyer:2014}, motivated by the topes and trees arising from a linkage matching field. 

\smallskip

A major theme throughout is to draw parallels between matching fields and multiple other combinatorial objects, see Figures \ref{fig:matching+field+classes}, \ref{fig:trianguloid+relationship} and \ref{fig:fields+stacks}, offering new methods to study the interplay between these objects.
We think that the new tools developed in this work can help to tackle the problems posed in~\cite{ArdilaBilley} which are subject of active research~\cite{GottiPolo:2018, CeballosPadrolSarmiento:2015,Santos:2013}. 
Furthermore, representing a triangulation of $\Delta_{n-1} \times \Delta_{d-1}$ by its set of degree vector pairs yields a new parameter space for the triangulations.
It was shown in \cite{Liu:2018} that the flip graph of triangulations is not connected. 
The correspondence with lattice points might allow a more detailed study of the flip graph of triangulations.

\subsection{Overview}

We define tope fields and related concepts in Section~\ref{sec:tope-fields}.
In particular, we introduce the relation with triangulations of products of two simplices and further requirements on the topes.

Section~\ref{sec:linkage+matching+fields}~is dedicated to our results on linkage matching fields.
Our fundamental construction for linkage matching fields is presented in Theorem~\ref{thm:constructed+topes}.
It leads to the resolution of two conjectures concerning Chow covectors in Theorem~\ref{thm:bijection-chow-covector-lattice-points} and Theorem~\ref{thm:lattice-points-Chow}.
The cryptomorphism between linkage matching fields and tope arrangements in Theorem~\ref{thm:cryptomorphism+linkage+mf} as well as the description of the flip graph of a linkage matching field in Theorem~\ref{thm:prodsimplicial+complex+flip+graph} are further consequences.

In Section~\ref{sec:pairs+of+lattice+points}, we deduce the reconstructability of a triangulation from a lattice point bijection in Theorem~\ref{thm:phi+injective}, analogously to the statement for Chow covectors.

We finish with the relation to sets of transversal matroids through a combinatorial Stiefel map in Section~\ref{sec:matching+stacks+transversal+matroids}.
This leads to several questions concerning the interplay between matching fields, triangulations and matroid subdivisions. 

\section{Tope fields} \label{sec:tope-fields}

Fix a pair $(n,d)$ of positive integers where $n \geq d$. 
We study bipartite graphs on two node sets $\lnoset$ and $\rnoset$, where $|\lnoset| = n$ and $|\rnoset|=d$.
Since they are defined on the same set of nodes, we will often identify them with their set of edges written as pairs of nodes.
The elements of $\lnoset$ are denoted by $\ell_1,\ldots, \ell_n$, the elements of $\rnoset$ by $r_1, \ldots, r_d$.
We refer to nodes in $\lnoset$ as \emph{left nodes} and nodes in $\rnoset$ as \emph{right nodes}.
The left degree vector is the tuple of node degrees of $\ell_1, \ldots, \ell_n$; we define the right degree vector analogously for the elements in $\rnoset$.

For two finite sets $A \subseteq B$ we denote the characteristic vector of $A$ in $B$ by $\unit{A}^B$, where we omit the reference set $B$ if it is clear from the context.

\begin{definition}
  Let $(v_1,\ldots, v_d)$ be a tuple of positive integers with $\sum_{i=1}^{d} v_i = k \leq n$.
	For a $k$-element subset $\sigma$ of $\lnoset$, we define a \emph{tope} of type $(v_1,\ldots, v_d)$ to be a bipartite graph whose right degree vector is its type and the left degree vector is $\unit{\sigma}$.
  An $(n,d)$-\emph{tope field} of type $(v_1,\ldots, v_d)$ is a set of topes $\cM = (M_{\sigma})$ with a unique tope $M_{\sigma}$ of type $\sigma$ for each $\sigma \in \binom{[n]}{k}$.
  The sum $k = \sum_{i=1}^{d} v_i$ is the \emph{thickness} of the tope field.
  If the thickness is $d$, the type is $\unit{[d]}$ and the tope field is a \emph{matching field}.
	If the thickness is $n$, the tope field has a single tope with left degree vector $\unit{[n]}$ that we call \emph{maximal}. 
\end{definition}

Note that our definition of topes differs slightly from the original definition in \cite{ArdilaDevelin:2009} as we insist that all right nodes must have positive degree.
The recent work~\cite{GalashinNenashevPostnikov2018} refers to them as semi-matchings and also allows topes to have isolated right nodes.
We shall see that these topes can be considered as lying `at the boundary'.

There is a natural arbitrariness in the role of $\lnoset$ and $\rnoset$. The previous definition is for a \emph{left} tope field, a \emph{right} tope field can be defined analogously for $|\lnoset| \leq |\rnoset|$. This distinction will occur when we deal with matching stacks. 
Note that a tope field can also be considered as a set of surjective functions $M_{\sigma}:\sigma \rightarrow R$ where $\left|M_{\sigma}^{-1}(\rno_i)\right| = v_i$.

A \emph{sub-tope field} is a tope field which consists of the induced subgraphs on $J \sqcup I$ for subsets $J \subseteq \lnoset$ and $I \subseteq \rnoset$. Note that in general the induced subgraphs on $J \sqcup I$ do not form a tope field.

Observe that tope fields generalise matching fields in the sense of \cite{SturmfelsZelevinsky:1993} and each graph is a tope in the sense of \cite{ArdilaDevelin:2009}.
Examples of matching and tope fields are given in Figures~\ref{fig:four+two+matching+field} and~\ref{fig:four+two+tope+field}.

\begin{figure}[htb]
  \begin{center}
	\resizebox{\textwidth}{!}{
  \begin{tikzpicture}[scale=0.6]

\bigraphfourtwocoord{0}{0}{1.2}{0.6}{2.5};
\draw[EdgeStyle] (v1) to (w1);
\draw[EdgeStyle] (v2) to (w2);
\bigraphfourtwonodes

\bigraphfourtwocoord{4}{0}{1.2}{0.6}{2.5};
\draw[EdgeStyle] (v1) to (w1);
\draw[EdgeStyle] (v3) to (w2);
\bigraphfourtwonodes

\bigraphfourtwocoord{8}{0}{1.2}{0.6}{2.5};
\draw[EdgeStyle] (v1) to (w1);
\draw[EdgeStyle] (v4) to (w2);
\bigraphfourtwonodes

\bigraphfourtwocoord{12}{0}{1.2}{0.6}{2.5};
\draw[EdgeStyle] (v2) to (w1);
\draw[EdgeStyle] (v3) to (w2);
\bigraphfourtwonodes

\bigraphfourtwocoord{16}{0}{1.2}{0.6}{2.5};
\draw[EdgeStyle] (v4) to (w1);
\draw[EdgeStyle] (v2) to (w2);
\bigraphfourtwonodes

\bigraphfourtwocoord{20}{0}{1.2}{0.6}{2.5};
\draw[EdgeStyle] (v4) to (w1);
\draw[EdgeStyle] (v3) to (w2);
\bigraphfourtwonodes
\end{tikzpicture}
	}
  \caption{A $(4,2)$-matching field.}
  \label{fig:four+two+matching+field}
	\end{center}
\end{figure}
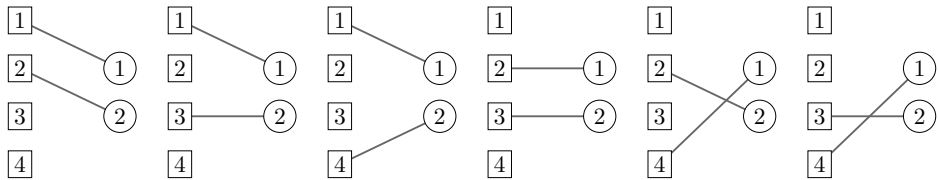

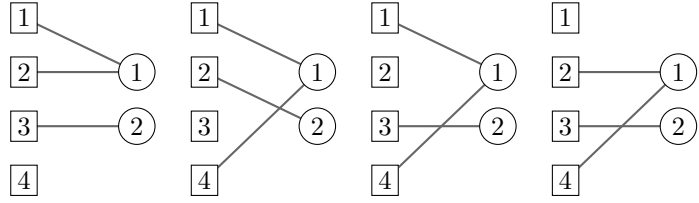
\begin{figure}[htb]
  \begin{center}
  \begin{tikzpicture}[scale=0.6]

\bigraphfourtwocoord{0}{0}{1.2}{0.6}{2.5};
\draw[EdgeStyle] (v1) to (w1);
\draw[EdgeStyle] (v2) to (w1);
\draw[EdgeStyle] (v3) to (w2);
\bigraphfourtwonodes

\bigraphfourtwocoord{4}{0}{1.2}{0.6}{2.5};
\draw[EdgeStyle] (v1) to (w1);
\draw[EdgeStyle] (v4) to (w1);
\draw[EdgeStyle] (v2) to (w2);
\bigraphfourtwonodes

\bigraphfourtwocoord{8}{0}{1.2}{0.6}{2.5};
\draw[EdgeStyle] (v1) to (w1);
\draw[EdgeStyle] (v4) to (w1);
\draw[EdgeStyle] (v3) to (w2);
\bigraphfourtwonodes

\bigraphfourtwocoord{12}{0}{1.2}{0.6}{2.5};
\draw[EdgeStyle] (v2) to (w1);
\draw[EdgeStyle] (v4) to (w1);
\draw[EdgeStyle] (v3) to (w2);
\bigraphfourtwonodes

\end{tikzpicture}
  \caption{A $(4,2)$-tope field of type $(2,1)$ with thickness 3.}
  \label{fig:four+two+tope+field}
	\end{center}
\end{figure}

\begin{example} \label{ex:coherent-matching-field}
  The most natural and well-behaved examples of matching fields are obtained from a generic matrix $A \in \RR^{n \times d}$.
	The minimal matchings in the complete bipartite graph $K_{n,d}$ weighted by the entries of $A$ give rise to a matching field.
  Such a matching field is called \emph{coherent}, cf. \cite{SturmfelsZelevinsky:1993}.
	If $A$ is not generic, one can slightly perturb the matrix to obtain a unique matching on each $d$-subset of $[n]$.
	Alternatively, one could just pick one minimal matching for each $d$-subset.
  However, the resulting matching field may be arbitrarily unstructured as one sees if $A$ is just the zero matrix.

  This idea can be extended to obtain \emph{coherent} tope fields in a similar fashion.
  We fix a vector $(v_1, \ldots, v_d) \in \ZZ_{>0}^d$ with $\sum_{i=1}^{d}v_i = k \leq n$ and a generic matrix $A \in \RR^{n \times d}$.
  We construct the matrix $\overline{A} \in \RR^{n \times k}$ by replacing the column indexed by $i$ in $A$ with $v_i$ copies of itself.
	Such a matrix gives rise to a complete bipartite graph $K_{n,k}$ with weights $\overline{A}$.
	The coherent matching field arising from $\overline{A}$ naturally yields a tope field of type $(v_1, \ldots, v_d)$ by setting $M_{\sigma}^{-1}(\rno_i)$ to the set of nodes in $\sigma$ adjacent to a copy of $r_i$ in the matching on $K_{n,k}$.
\end{example}

The process of duplicating the columns in the definition of a coherent tope field motivates the next construction.

\begin{example} \label{ex:increasing-splitting}
  An $(n,d)$-tope field $\cM = (M_{\sigma})$ of type $(v_1,\ldots, v_d)$ gives rise to an $(n,k)$-matching field $\cN$ where $k = \sum_{i=1}^{d} v_i$.
	For each node $\rno_i \in \rnoset$ we introduce $v_i$ nodes $\rno_i^{(1)}, \ldots, \rno_i^{(v_i)}$.
	Let $j^{(1)} < \ldots < j^{(v_i)}$ be the increasing list of indices denoting the elements in $M_{\sigma}^{-1}(\rno_i)$.
	By setting $N_{\sigma}(\lno_{j^{(t)}}) = \rno_i^{(t)}$ for each $t \in [v_i]$ and all $i \in [d]$, we obtain a matching field $\cN = (N_{\sigma})$.
	We call this matching field the \emph{increasing splitting} of the tope field.

	Observe that one could also consider partial splitting from a coarser to a finer tope field.
	Furthermore, note that the splitting depends on the ordering of the split copies of the nodes in $\rnoset$.
	This construction can be seen as a `refinement' of the tope field, analogous to a refinement of a polyhedral subdivision in~\cite[Definition 2.3.8]{DeLoeraRambauSantos}.
	In particular, the linkage covectors, see Definition \ref{def:linkage+covector}, of the increasing splitting can be seen as full dimensional cells in a staircase triangulation of $\Dprod{n-1}{d-1}$.
	
	Increasing splitting is analogous to the process of polarisation from commutative algebra, in which a monomial ideal can be transformed to a squarefree monomial ideal.
	This is done by replacing a single variable of degree $d$ in a generator by the product of $d$ distinct variables, giving an ideal with far simpler combinatorics.	
\end{example}

\subsection{An important example} \label{sec:matching+fields+from+triangulations}

For the polyhedral background we refer to \cite{DeLoeraRambauSantos}. 

The construction in \cite{OhYoo-ME:2013} connecting matching fields and triangulations of a product of two simplices motivates us to investigate matching fields and their connection to polyhedral constructions further.

Recall that the maximal simplices in a triangulation of $\Delta_{n-1} \times \Delta_{d-1}$ are given by a set of spanning trees on the bipartite node set $\lnoset \sqcup \rnoset$ with $|\lnoset| = n, |\rnoset| = d$, which fulfil the axioms given by Ardila \& Billey \cite{ArdilaBilley}.
A simplex with vertices $(\unit{j_1},\unit{i_1}),\ldots,(\unit{j_k},\unit{i_k}) \in \Delta_{n-1} \times \Delta_{d-1}$ corresponds to the bipartite graph with edges $(\lno_{j_1},\rno_{i_1}),\ldots,(\lno_{j_k},\rno_{i_k})$.

\begin{proposition}[{\cite[Proposition 7.2]{ArdilaBilley}}] \label{prop:char-triang} \leavevmode
  A set of trees encodes the maximal simplices of a triangulation of $\Dprod{n-1}{d-1}$ if and only if:
  \begin{enumerate} 
  \item Each tree is spanning.
  \item For each tree $G$ and each edge $e$ of $G$, either $G - e$ has an isolated node or there is another tree $H$ containing $G - e$.
  \item If two trees $G$ and $H$ contain perfect matchings on $J \sqcup I$ for $J \subseteq \lnoset$ and $I \subseteq \rnoset$ with $|J| = |I|$ then the matchings agree. 
  \end{enumerate}
\end{proposition}
We wish to study the connection to triangulations of $\Delta_{n-1} \times \Delta_{d-1}$, therefore all trees we refer to will be spanning trees of the complete bipartite graph on $L \sqcup R$, unless otherwise stated.

Most of our arguments are independent of the embedding of the product of simplices. For the next proposition we choose the canonical embedding
\[
 \Delta_{n-1} \times \Delta_{d-1} = \conv\{(\unit{j},\unit{i}) \mid j \in [n], i \in [d]\} \subseteq \RR^{n+d} \enspace .
\]
Oh \& Yoo introduced in \cite{OhYoo-ME:2013} the \emph{extraction method} which collects the set of all partial matchings occurring in the trees encoding the triangulation. 
The fact that we obtain a matching field by taking all matchings of size $d$ occurring in the trees can be deduced from the following polyhedral construction.

\begin{proposition} \label{prop:existence-matching}
  Let $\Sigma$ be a triangulation of $\Dprod{d-1}{d-1}$. Then the bipartite graph $G$ corresponding to the minimal cell (with respect to inclusion) containing the barycentre $g = (\frac{1}{d},\ldots,\frac{1}{d})$ of $\Dprod{d-1}{d-1}$ is a perfect matching on $[d] \sqcup [d]$. 
\end{proposition}
\begin{proof}
By the first condition of Proposition \ref{prop:char-triang}, the bipartite graph $G$ is a subgraph of a tree, and therefore is a forest.
	Each point in the cell corresponding to $G$ is a unique convex combination of its vertices, in particular $g$.
	Define $\lambda$ to be the weight function which assigns to each edge of $G$ its coefficient in the representation of $g$.
	By minimality, $\lambda \neq 0$ for all the edges.
The graph $G$ contains a node of degree $1$ as it is a forest.
	The weight of the incident edge $e$ has to be $\frac{1}{d}$ as it determines the coordinate corresponding to the node of degree $1$.
	Therefore, the other node incident with $e$ has degree $1$ as well.
	By induction, this implies the claim. 
\end{proof}

Iterating the construction of Proposition \ref{prop:existence-matching} over all faces of $\Dprod{n-1}{d-1}$ of the form
\[
 \conv\{\unit{j} \mid j \in \sigma \} \times \Delta_{d-1} \mbox{ for } \sigma \in \binom{[n]}{d}
\]
produces a matching field.

\subsection{Compatibility, Trees and Topes}

	Arbitrary matching fields have very little structure, hence we shall study properties of matching fields which occur in connection to polyhedral constructions.
	The third condition in Proposition~\ref{prop:char-triang} motivates the following notion which was coined in the context of tropical oriented matroids~\cite{ArdilaDevelin:2009}.

\begin{definition}
  Two forests $F_1$ and $F_2$ on the same node set $\lnoset \sqcup \rnoset$ are \emph{compatible} if for all subsets $J \subseteq \lnoset$ and $I \subseteq \rnoset$ such that $F_1$ and $F_2$ contain perfect matchings on $J \sqcup I$, those perfect matchings are equal.
  Otherwise, $F_1$ and $F_2$ are \emph{incompatible}.
\end{definition}
Note that we mainly apply this definition to matchings, topes and trees.
Furthermore, the incompatibility is often certified by an alternating cycle formed by two different matching on the same node set. 
The next lemma already occurs in \cite{OhYoo-ME:2013} but we give a proof to clarify our terminology.

\begin{lemma}[{\cite[Lemma 3.7]{OhYoo-ME:2013}}] \label{lem:unique-rdv-topes}
  Let $T_1$ and $T_2$ be distinct topes defined on $\lnoset \sqcup \rnoset$. If they have the same left and right degree vector, then they are incompatible.  
\end{lemma}
\begin{proof}
  Consider the symmetric sum of the edges of $T_1$ and $T_2$. Direct the edges of $T_1$ from $\lnoset$ to $\rnoset$ and of $T_2$ conversely.
	In the resulting graph, the indegree and the outdegree of each node are equal.
	Hence, the graph contains a directed cycle.
	This consists of two different partial matchings on the same node set in $T_1$ and $T_2$. 
\end{proof}

There is an analogous statement for trees.

\begin{lemma}[{\cite[Lemma 12.8]{Postnikov:2009}}] \label{lem:degree-vectors-trees}
  Let $T_1$ and $T_2$ be two different spanning trees on $\lnoset \sqcup \rnoset$. If they have the same left degree vector or the same right degree vector then they are incompatible.
\end{lemma}

The next statement introduces a simple but crucial construction of topes with prescribed degree vector from a tree. 

\begin{proposition} \label{prop:unique+contained+tope}
  Let $G$ be a spanning tree on $\lnoset \sqcup \rnoset$ with right degree vector $v = (v_1, \ldots, v_d)$.
  For each $\rno_{k} \in \rnoset$, there is a unique maximal tope with right degree vector $v - \unit{[d] \setminus \{k\}}$ contained in $G$.
\end{proposition}
\begin{proof}
  For each path from $\rno_{k}$ to another node in $\rnoset$ remove the last edge in that path.
  The resulting graph is the desired tope.
  To show uniqueness, suppose there is another tope with right degree vector $v - \unit{[d] \setminus \{k\}}$ contained in $G$.
  By Lemma \ref{lem:unique-rdv-topes}, the two topes are incompatible.
  This gives a contradiction as the union of the topes contains a cycle and $G$ does not.
\end{proof}

\begin{corollary} \label{cor:neighbour-trees}
  Let $T_1$ and $T_2$ be two compatible spanning trees where the first has right degree vector $v$ and the second $v + \unit{p} - \unit{q}$.
  Then their intersection contains a maximal tope with right degree vector $v - \unit{[d] \setminus \{p\}}$.
	Furthermore, if $\lno_s$ has degree 1 in both $T_1$ and $T_2$ then the trees agree on the edge adjacent to $\lno_s$.
\end{corollary}
\begin{proof}
  Proposition~\ref{prop:unique+contained+tope} ensures that both trees contain a tope with right degree vector $v - \unit{[d] \setminus \{p\}}$.
  Those topes agree as a consequence of Lemma~\ref{lem:unique-rdv-topes} because of the compatibility of $T_1$ and $T_2$.
	If $\lno_s$ has degree 1 in both trees, the edge adjacent to it must be contained in the unique tope in their intersection.
\end{proof}

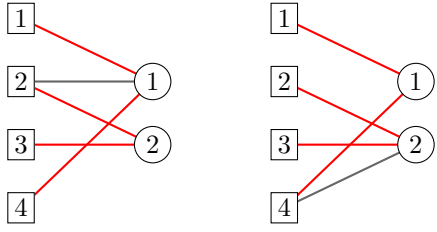
\begin{figure}%
\begin{tikzpicture}[scale=0.7]
	\bigraphfourtwocoord{0}{0}{1.2}{0.6}{2.5};
\draw[EdgeStyle,red] (v1) to (w1);
\draw[EdgeStyle] (v2) to (w1);
\draw[EdgeStyle,red] (v4) to (w1);
\draw[EdgeStyle,red] (v2) to (w2);
\draw[EdgeStyle,red] (v3) to (w2);
\bigraphfourtwonodes

\bigraphfourtwocoord{5}{0}{1.2}{0.6}{2.5};
\draw[EdgeStyle,red] (v1) to (w1);
\draw[EdgeStyle,red] (v4) to (w1);
\draw[EdgeStyle] (v4) to (w2);
\draw[EdgeStyle,red] (v2) to (w2);
\draw[EdgeStyle,red] (v3) to (w2);
\bigraphfourtwonodes
\end{tikzpicture}%
\caption{Two compatible trees with right degree vectors $(3,2)$ and $(2,3)$ respectively. The common (red) edges form a tope with right degree vector $(2,2)$.}%
\label{fig:tree+intersection}%
\end{figure}

Proposition \ref{prop:unique+contained+tope} and Corollary \ref{cor:neighbour-trees} emphasise the structural relationship between topes and trees, in particular how to recover one from the other.
Figure \ref{fig:tree+intersection} shows an example of recovering topes from intersections of trees.
However, as we shall later see, we can take unions of topes to recover trees.

We can build on these results to find even stronger local conditions on compatible trees.
The following lemma captures a combinatorial analogue of certain geometric properties explained in Example~\ref{ex:hyperplane+sector}.

\begin{lemma} \label{lem:abstract-containment-sector}
  Let $T_1$ and $T_2$ be two compatible spanning trees on $\lnoset \sqcup \rnoset$ with right degree vectors $v$ and $v + \unit{p} - \unit{q}$ respectively.
  If $(\lno_s,\rno_p)$ is an edge in $T_1$ then it is also an edge in $T_2$.
  Furthermore, the degree $\delta_1$ of $\lno_s$ in $T_1$ is bigger or equal to its degree $\delta_2$ in $T_2$.
\end{lemma}
\begin{proof}
  By Corollary~\ref{cor:neighbour-trees}, the intersection $T_1 \cap T_2$ contains a tope with right degree vector $v - \unit{[d] \setminus \{p\}}$.
 The first claim follows directly from this, as if $(\lno_s,\rno_p)$ is an edge in $T_1$, it is contained in this tope, and therefore $T_2$. 

For the second claim, we define an auxiliary graph $H$ as follows.
We take the union of the graphs $T_1 \cup T_2$ and make the following alterations.
We direct the edges of $T_1 \setminus T_2$ from $\rnoset$ to $\lnoset$ and direct the edges of $T_2 \setminus T_1$ from $\lnoset$ to $\rnoset$.
Finally we contract the remaining undirected edges, those in the intersection of $T_1$ and $T_2$.
Note that these undirected edges form a spanning forest on $\lnoset \sqcup \rnoset$.
The resulting graph is our auxiliary graph $H$, whose nodes correspond to connected components of $T_1 \cap T_2$.
To each node $v$, we attach the label $\phi(v) \subseteq \lnoset \sqcup \rnoset$, the subset of nodes contained in the connected component that has been contracted to that node.

As the node $r_i \in \rnoset\setminus \{r_p,r_q\}$ has the same degree in $T_1,T_2$, every node $v$ such that $\rno_p, \rno_q \notin \phi(v)$ has in-degree equal to out-degree.
This implies for every node $v$ such that $r_p \notin \phi(v)$, we have an outgoing arc.
 Assume that $\delta_2 > \delta_1$.
 Then an edge in $T_2$ incident with $\lno_s$ gives rise to an out-going arc from the node containing $\rno_p$ in $H$.
  Hence, each node in $H$ has an out-going arc which implies the existence of a cycle in this auxiliary graph.
  This however contradicts the compatibility of $T_1$ and $T_2$. 
\end{proof}

\begin{example} \label{ex:hyperplane+sector}
	The intuition behind Lemma \ref{lem:abstract-containment-sector} can be seen via arrangements of tropical hyperplanes.
	A tropical hyperplane $H$ is a translation of the normal fan of $\Delta_{d-1}$, decomposing $\RR^{d-1}$ into $d$ sectors $S^{(1)},\dots,S^{(d)}$ labelled by the vertices of $\Delta_{d-1}$.
        Develin and Sturmfels~\cite{DevelinSturmfels:2004} showed that the covectors (`types' in their terminology) labelling the cells of a tropical hyperplane arrangement also describe a regular subdivision of $\Delta_{n-1} \times \Delta_{d-1}$.
        In a generic arrangement, this yields a set of trees encoding the arrangement.
	They can be extracted from the arrangement via the zero dimensional cells: the corresponding tree contains the edge $(\lno_j,\rno_i)$ if the cell is contained in $S_j^{(i)}$, sector $i$ of hyperplane $j$.
	
	Consider the example depicted in Figure \ref{fig:hyperplane+sector}.
	The trees $T_1$ and $T_2$ are the covectors of the apexes of the corresponding hyperplanes $H_1$ and $H_2$.
	The edge $(\lno_1,\rno_1)$ in $T_1$ implies that the corresponding cell is in the sector $S_1^{(1)}$.
	As $\rno_1$ has a larger degree in $T_2$, walking to the apex of $H_2$ requires us to move further in the `$1$-direction', and therefore we do not leave $S_1^{(1)}$.
	Lemma~\ref{lem:abstract-containment-sector} is a purely combinatorial description of this behaviour that also covers the non-regular case.
	This will later motivate our definition of a combinatorial analogue of tropical hyperplane sectors.
\end{example}

\begin{figure}%
\begin{tikzpicture}
	\node[BoxVertex] (1) at (0.5,1) {1};
	\node[BoxVertex] (2) at (0,0) {2};
	\draw (2) -- (0,-1);
	\draw (2) -- (-1,0.6);
	\draw (2) -- (2,1.2);
	\draw (1) -- (0.5,-1);
	\draw (1) -- (-1,1.9);
	\draw (1) -- (2,1.9);
	
	\node[BoxVertex] (1) at (6,1) {1};
	\node[BoxVertex] (2) at (6,0) {2};
	\node[VertexStyle] (11) at (7,1.5) {1};
	\node[VertexStyle] (22) at (7,0.5) {2};
	\node[VertexStyle] (33) at (7,-0.5) {3};
	\draw (2) -- (11);
	\draw (2) -- (22);
	\draw (2) -- (33);
	\draw (1) -- (11);
	\node at (6.5,-1) {$T_2$};

	\node[BoxVertex] (1) at (4,1) {1};
	\node[BoxVertex] (2) at (4,0) {2};
	\node[VertexStyle] (11) at (5,1.5) {1};
	\node[VertexStyle] (22) at (5,0.5) {2};
	\node[VertexStyle] (33) at (5,-0.5) {3};
	\draw (1) -- (11);
	\draw (1) -- (22);
	\draw (1) -- (33);
	\draw (2) -- (33);
	\node at (4.5,-1) {$T_1$};
	
	\node[BoxVertex] (1) at (8,1) {1};
	\node[BoxVertex] (2) at (8,0) {2};
	\node[VertexStyle] (11) at (9,1.5) {1};
	\node[VertexStyle] (22) at (9,0.5) {2};
	\node[VertexStyle] (33) at (9,-0.5) {3};
	\draw (2) -- (22);
	\draw (2) -- (33);
	\draw (1) -- (11);
	\draw (1) -- (22);
	\node at (8.5,-1) {$T_3$};
\end{tikzpicture}
\caption{An arrangement of two tropical hyperplanes in $\RR^2$ and their corresponding trees.
  $T_1,T_2$ correspond to the apices of the hyperplanes, while $T_3$ is the tree corresponding to their intersection.}%
\label{fig:hyperplane+sector}%
\end{figure}
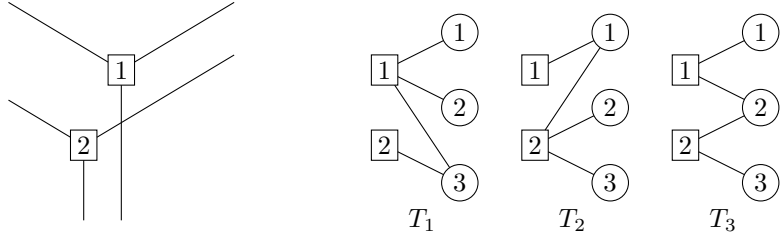

\section{Linkage matching fields} \label{sec:linkage+matching+fields}

Recall from \cite{SturmfelsZelevinsky:1993} the definition of the linkage property of a matching field.

\begin{definition} \label{def:linkage+axiom}
Let $\cM = (M_{\sigma})$ be an $(n,d)$-matching field where $n \geq d$. We say $\cM$ is \emph{linkage} if the following \emph{linkage axiom} holds: for every $\rno_i \in \rnoset$ and $(d+1)$-subset $\tau \subseteq \lnoset$ there exist two distinct $\lno_j,\lno_{j'} \in \tau$ such that the matchings $M_{\tau \setminus \{\lno_j\}}$ and $M_{\tau \setminus \{\lno_{j'}\}}$ agree everywhere other than on $\rno_i$.
\end{definition}

\begin{remark}
When considering a matching field as a set of bipartite graphs, it will be useful to differentiate between \emph{left linkage} and \emph{right linkage}. The previous definition is for left linkage as it describes how matchings on varying subsets of $\lnoset$ are linked. Right linkage is defined analogously for matchings of right matching fields, where $|\rnoset| \geq |\lnoset|$ and one ranges over all $(|\lnoset|+1)$-subsets of $\rnoset$. 
\end{remark}

Reformulating the conditions from \cite[Theorem 2.4(3)]{SturmfelsZelevinsky:1993} and \cite[Corollary 2.12]{SturmfelsZelevinsky:1993} yields the following.

  \begin{lemma} \label{lem:linkage+and+tree}
   The linkage axiom is equivalent to the following condition: 
let $\tau$ be a $(d+1)$-element subset of $\lnoset$. Then the union of all matchings on $\tau \sqcup \rnoset$  is a tree where all right nodes have degree 2.
  \end{lemma}

Note that the formerly described trees can also be characterised in terms of `support sets' in the sense of \cite[Proposition 2.6]{FinkRincon:2015}.
This yields another description of matching linkage covectors, see Definition~\ref{def:linkage+covector}, based on degree conditions.

The linkage axiom can be formulated in even more ways, the following of which we take advantage of:

\begin{lemma}\cite[Theorem 2b]{BernsteinZelevinsky:1993} \label{lem:subset+linkage}
The linkage axiom is equivalent to the following property: for every two distinct $d$-subsets $\sigma,\sigma' \subset \lnoset$ there exists $\lno_{j'} \in \sigma' \setminus \sigma$ such that if $(\lno_{j'},\rno_i)$ is an edge of $M_{\sigma'}$ and $(\lno_j,\rno_i)$ is an edge of $M_{\sigma}$ then the matchings $M_{\sigma}$ and $M_{\sigma \setminus \{\lno_j\} \cup \{\lno_{j'}\}}$ agree everywhere other than on $\rno_i$.
\end{lemma}

The linkage property implies the compatibility of the occurring matchings as the next statement shows.

\begin{proposition}[Weak compatibility] \label{prop:weak-compatibility}
  Let $\cM$ be a linkage matching field of type $(n,d)$ with $n \geq d$ and let $\sigma, \sigma'$ be two $d$-element subsets of $\lnoset$. If there are subsets $P \subseteq \sigma \cap \sigma'$ and $Q \subseteq \rnoset$ with $|P| = |Q|$ such that $M_{\sigma}|_{P \sqcup Q}$ and $M_{\sigma'}|_{P \sqcup Q}$ are perfect matchings then those matchings agree.
\end{proposition}
\begin{proof}
  The claim follows by induction on the size of the intersection $\sigma \cap \sigma'$.
  For $\sigma = \sigma'$ the claim is just the fact that there is exactly one matching per $d$-element subset in the matching field.
  Otherwise, by Lemma~\ref{lem:subset+linkage}, there is an $\lno_{j'} \in \sigma'\setminus \sigma$ with certain properties.
  Since $\lno_{j'} \not\in P$, the node $\rno_i$ adjacent to $\lno_{j'}$ in $M_{\sigma'}$ is not an element of $Q$.
  Hence, the node $\lno_j$ adjacent to $\rno_i$ in $M_{\sigma}$ is not in $P$.
  By Lemma~\ref{lem:subset+linkage}, the matching on $\sigma'' = \sigma \setminus \lno_j \cup \{\lno_{j'}\}$ is uniquely defined and also contains $M_{\sigma}|_{P \sqcup Q}$.
  Furthermore, $|\sigma'' \cap \sigma'| = |\left(\sigma \setminus \lno_j \cup \{\lno_{j'}\}\right) \cap \sigma'| = |\sigma \cap \sigma'|+1$.
  This concludes the proof by induction.
\end{proof}

Observe that compatibility is a weaker condition than linkage. Figure~\ref{fig:compatible+not+linkage} shows an example of a matching field whose matchings are pairwise compatible but do not satisfy linkage.

\begin{figure}%
\begin{tikzpicture}[scale=0.6]

\node[BoxVertex] (1) at (0,1.5){1};
\node[BoxVertex] (2) at (0,0.5){2};
\node[BoxVertex] (3) at (0,-0.5){3};
\node[BoxVertex] (4) at (0,-1.5){4};
\node[VertexStyle] (11) at (2,1){1};
\node[VertexStyle] (22) at (2,0){2};
\node[VertexStyle] (33) at (2,-1){3};

\draw[EdgeStyle] (1) to (11);
\draw[EdgeStyle] (2) to (22);
\draw[EdgeStyle] (3) to (33);

\node[BoxVertex] (1) at (4,1.5){1};
\node[BoxVertex] (2) at (4,0.5){2};
\node[BoxVertex] (3) at (4,-0.5){3};
\node[BoxVertex] (4) at (4,-1.5){4};
\node[VertexStyle] (11) at (6,1){1};
\node[VertexStyle] (22) at (6,0){2};
\node[VertexStyle] (33) at (6,-1){3};

\draw[EdgeStyle] (1) to (11);
\draw[EdgeStyle] (2) to (22);
\draw[EdgeStyle] (4) to (33);

\node[BoxVertex] (1) at (8,1.5){1};
\node[BoxVertex] (2) at (8,0.5){2};
\node[BoxVertex] (3) at (8,-0.5){3};
\node[BoxVertex] (4) at (8,-1.5){4};
\node[VertexStyle] (11) at (10,1){1};
\node[VertexStyle] (22) at (10,0){2};
\node[VertexStyle] (33) at (10,-1){3};

\draw[EdgeStyle] (4) to (11);
\draw[EdgeStyle] (1) to (22);
\draw[EdgeStyle] (3) to (33);

\node[BoxVertex] (1) at (12,1.5){1};
\node[BoxVertex] (2) at (12,0.5){2};
\node[BoxVertex] (3) at (12,-0.5){3};
\node[BoxVertex] (4) at (12,-1.5){4};
\node[VertexStyle] (11) at (14,1){1};
\node[VertexStyle] (22) at (14,0){2};
\node[VertexStyle] (33) at (14,-1){3};

\draw[EdgeStyle] (2) to (11);
\draw[EdgeStyle] (3) to (22);
\draw[EdgeStyle] (4) to (33);

\end{tikzpicture}%
\caption{A $(4,3)$-matching field whose matchings are compatible. Note that the matching on $\{\lno_1,\lno_3,\lno_4\}$ does not agree with any other matching on two edges, and therefore the matching field does not satisfy the linkage property.
Equivalently, the union of the matchings contains a cycle and, hence, it is not a tree. }%
\label{fig:compatible+not+linkage}%
\end{figure}
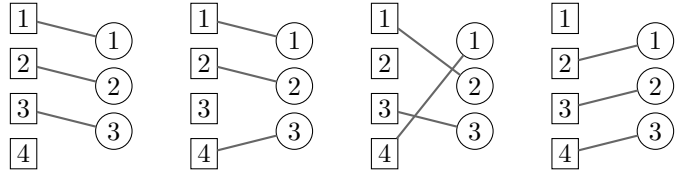

We have now gathered the necessary tools to construct a linkage tope field from a linkage matching field.

\subsection{Tope fields from matching fields} \label{sec:tope+from+matching}

\begin{definition} \label{def:linkage+covector}
  We say that an $(n,d)$-tope field of type $(v_1,\ldots, v_d)$ with thickness $k$ is \emph{linkage} if for all $(k+1)$-subsets $\tau$ of $\lnoset$, the union of the topes on $\tau$ is a tree. 
  Such a tree is a \emph{(tope) linkage covector}.
  In particular, an $(n,d)$-matching field is \emph{linkage} if for all $(d+1)$-subsets $\tau$ of $\lnoset$, the union of the matchings on $\tau$ is a tree. 
  We call this tree a \emph{(matching) linkage covector}.
\end{definition}

\begin{remark}
 The linkage covectors of a matching field are essentially the same as linkage trees defined in \cite{SturmfelsZelevinsky:1993}.
One can transform a matching linkage tree to a linkage covector by replacing the edge $(j,j')$ with label $i$ by the edges $(\lno_j, \rno_i), (\lno_{j'}, \rno_i)$.
Linkage trees will play a role in results on the flip graph of a matching field in Section \ref{sec:flip+graph}.
\end{remark}

\begin{example}
Consider the matching field and tope field given in Figures~\ref{fig:four+two+matching+field} and~\ref{fig:four+two+tope+field}. Both of these are linkage with the corresponding linkage covectors given in Figure~\ref{fig:four+two+linkage+covectors}.

\begin{figure}[htb]
  \begin{center}
	\resizebox{\textwidth}{!}{
  \begin{tikzpicture}[scale=0.6]

\bigraphfourtwocoord{0}{0}{1.2}{0.6}{2.5};
\draw[EdgeStyle] (v1) to (w1);
\draw[EdgeStyle] (v2) to (w1);
\draw[EdgeStyle] (v2) to (w2);
\draw[EdgeStyle] (v3) to (w2);
\bigraphfourtwonodes

\bigraphfourtwocoord{4}{0}{1.2}{0.6}{2.5};
\draw[EdgeStyle] (v1) to (w1);
\draw[EdgeStyle] (v4) to (w1);
\draw[EdgeStyle] (v2) to (w2);
\draw[EdgeStyle] (v4) to (w2);
\bigraphfourtwonodes

\bigraphfourtwocoord{8}{0}{1.2}{0.6}{2.5};
\draw[EdgeStyle] (v1) to (w1);
\draw[EdgeStyle] (v4) to (w1);
\draw[EdgeStyle] (v4) to (w2);
\draw[EdgeStyle] (v3) to (w2);
\bigraphfourtwonodes

\bigraphfourtwocoord{12}{0}{1.2}{0.6}{2.5};
\draw[EdgeStyle] (v2) to (w1);
\draw[EdgeStyle] (v4) to (w1);
\draw[EdgeStyle] (v2) to (w2);
\draw[EdgeStyle] (v3) to (w2);
\bigraphfourtwonodes

\bigraphfourtwocoord{20}{0}{1.2}{0.6}{2.5};
\draw[EdgeStyle] (v1) to (w1);
\draw[EdgeStyle] (v2) to (w1);
\draw[EdgeStyle] (v4) to (w1);
\draw[EdgeStyle] (v2) to (w2);
\draw[EdgeStyle] (v3) to (w2);
\bigraphfourtwonodes

\end{tikzpicture}
	}
  \caption{The four matching linkage covectors from the $(4,2)$-matching field and the one tope linkage covector from the $(4,2)$-tope field.}
  \label{fig:four+two+linkage+covectors}
	\end{center}
\end{figure}
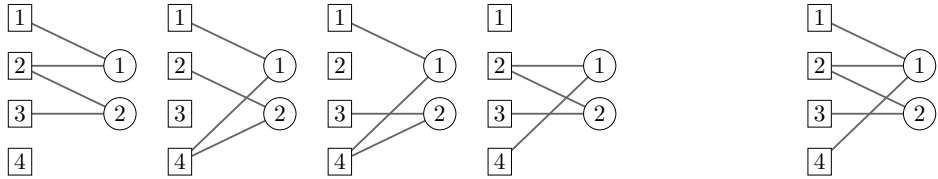

\end{example}

\begin{example} \label{ex:linkage+extracted+matching+field}
  We continue to highlight the relationship to triangulations of the polytope~$\Delta_{n-1}~\times~\Delta_{d-1}$ started in Subsection~\ref{sec:matching+fields+from+triangulations}.
	Let $\tau \subseteq \lnoset$ be a $(d+1)$-subset.
	The triangulation induced on the product
  \[
  \conv\{\unit{j} \mid \lno_j \in \tau\} \times \Delta_{d-1}  \enspace ,
  \]
  which is a face of $\Delta_{n-1} \times \Delta_{d-1}$, contains a unique maximal simplex whose corresponding tree has right degree vector $2 \cdot \unit{[d]}$ by \cite[Theorem 12.9]{Postnikov:2009}.
  This is the linkage covector of the matching field on $\tau$ as can be seen from Proposition~\ref{prop:unique+contained+tope}. This follows as it contains all the matchings on the $d$-subsets of $\tau$ and, hence, their union is just this tree.
  Hence, the matching field derived from a triangulation is linkage, as also stated in~\cite{OhYoo-ME:2013}. 
  Restricting to regular subdivisions, this implies that coherent matching fields are linkage, see Example~\ref{ex:coherent-matching-field} and~\cite{SturmfelsZelevinsky:1993}.

  \smallskip

  A matching field derived from a non-regular triangulation is depicted in Figure~\ref{fig:nonregular+subdivision}.
	The picture on the left shows the non-regular fine mixed subdivision of $6\Delta_2$ which corresponds to a non-regular triangulation of $\Dprod{5}{2}$ by the Cayley trick, both of which can be found in~\cite{DeLoeraRambauSantos}.
	Every cell contains an `upward' unit simplex, in particular cells who do not share a facet with the boundary have their upward simplices coloured grey.
	The right side shows the bipartite graphs describing the maximal simplices corresponding to the grey upward simplices on the left.
	They are in bijection with the lattice points given by the right degree vectors of the trees on the right.
	Inner cells that do not share a facet with the boundary correspond to those trees whose right degree vector does not contain a $1$ entry.
\end{example}

\begin{figure}%
\centering
\includegraphics[width=0.48\textwidth]{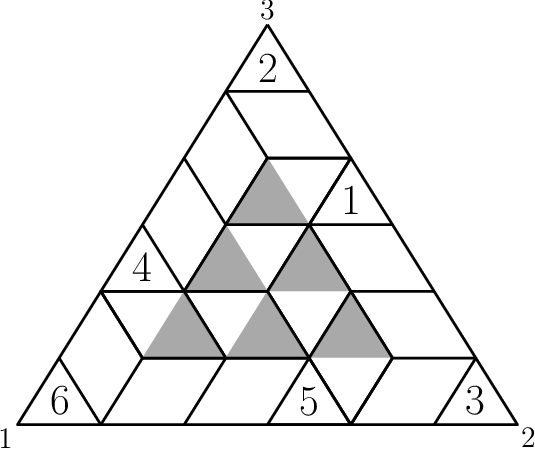}
\quad
    \scalebox{0.57}{ \begin{tikzpicture}[scale=0.6]

\newcommand{\bigraphsixthreecoord}[5]{
\coordinate (v1) at (#1,#2+2.5*#3);
\coordinate (v2) at (#1,#2+1.5*#3);
\coordinate (v3) at (#1,#2+0.5*#3);
\coordinate (v4) at (#1,#2-0.5*#3);
\coordinate (v5) at (#1,#2-1.5*#3);
\coordinate (v6) at (#1,#2-2.5*#3);

\coordinate (w1) at (#1+#5,#2+2*#4);
\coordinate (w2) at (#1+#5,#2);
\coordinate (w3) at (#1+#5,#2-2*#4);
}

\newcommand{\bigraphsixthreenodes}{
\node[BoxVertex] (1) at (v1){1};
\node[BoxVertex] (2) at (v2){2};
\node[BoxVertex] (3) at (v3){3};
\node[BoxVertex] (4) at (v4){4};
\node[BoxVertex] (5) at (v5){5};
\node[BoxVertex] (6) at (v6){6};

\node[VertexStyle] (11) at (w1){1};
\node[VertexStyle] (22) at (w2){2};
\node[VertexStyle] (33) at (w3){3};
}

\bigraphsixthreecoord{0}{5.1}{0.8}{0.6}{2.5};
\draw[EdgeStyle] (v1) to (w1);
\draw[EdgeStyle] (v2) to (w1);
\draw[EdgeStyle] (v1) to (w3);
\draw[EdgeStyle] (v2) to (w2);
\draw[EdgeStyle] (v4) to (w2);
\draw[EdgeStyle] (v3) to (w3);
\draw[EdgeStyle] (v5) to (w3);
\draw[EdgeStyle] (v6) to (w3);
\bigraphsixthreenodes

\bigraphsixthreecoord{-3}{0}{0.8}{0.6}{2.5};
\draw[EdgeStyle] (v1) to (w1);
\draw[EdgeStyle] (v2) to (w1);
\draw[EdgeStyle] (v3) to (w1);
\draw[EdgeStyle] (v2) to (w2);
\draw[EdgeStyle] (v4) to (w2);
\draw[EdgeStyle] (v3) to (w3);
\draw[EdgeStyle] (v5) to (w3);
\draw[EdgeStyle] (v6) to (w3);
\bigraphsixthreenodes

\bigraphsixthreecoord{3}{0}{0.8}{0.6}{2.5};
\draw[EdgeStyle] (v1) to (w1);
\draw[EdgeStyle] (v6) to (w2);
\draw[EdgeStyle] (v3) to (w1);
\draw[EdgeStyle] (v2) to (w2);
\draw[EdgeStyle] (v4) to (w2);
\draw[EdgeStyle] (v3) to (w3);
\draw[EdgeStyle] (v5) to (w3);
\draw[EdgeStyle] (v6) to (w3);
\bigraphsixthreenodes

\bigraphsixthreecoord{-6}{-5.1}{0.8}{0.6}{2.5};
\draw[EdgeStyle] (v1) to (w1);
\draw[EdgeStyle] (v2) to (w1);
\draw[EdgeStyle] (v4) to (w2);
\draw[EdgeStyle] (v6) to (w2);
\draw[EdgeStyle] (v3) to (w1);
\draw[EdgeStyle] (v4) to (w1);
\draw[EdgeStyle] (v5) to (w3);
\draw[EdgeStyle] (v6) to (w3);
\bigraphsixthreenodes

\bigraphsixthreecoord{0}{-5.1}{0.8}{0.6}{2.5};
\draw[EdgeStyle] (v1) to (w1);
\draw[EdgeStyle] (v2) to (w1);
\draw[EdgeStyle] (v3) to (w1);
\draw[EdgeStyle] (v2) to (w2);
\draw[EdgeStyle] (v4) to (w2);
\draw[EdgeStyle] (v6) to (w2);
\draw[EdgeStyle] (v5) to (w3);
\draw[EdgeStyle] (v6) to (w3);
\bigraphsixthreenodes

\bigraphsixthreecoord{6}{-5.1}{0.8}{0.6}{2.5};
\draw[EdgeStyle] (v1) to (w1);
\draw[EdgeStyle] (v2) to (w2);
\draw[EdgeStyle] (v3) to (w1);
\draw[EdgeStyle] (v5) to (w2);
\draw[EdgeStyle] (v4) to (w2);
\draw[EdgeStyle] (v3) to (w3);
\draw[EdgeStyle] (v5) to (w3);
\draw[EdgeStyle] (v6) to (w2);
\bigraphsixthreenodes

\end{tikzpicture} }
\caption{The maximal linkage covectors of a matching field extracted from the interior cells of a non-regular subdivision of $6\Delta_2$ (example from \cite[Figure 9.53]{DeLoeraRambauSantos}).
Each cell contains a unique grey simplex, the interior one being in bijection with the lattice points of the simplex $2\Delta_2$.}%
\label{fig:nonregular+subdivision}%
\end{figure}
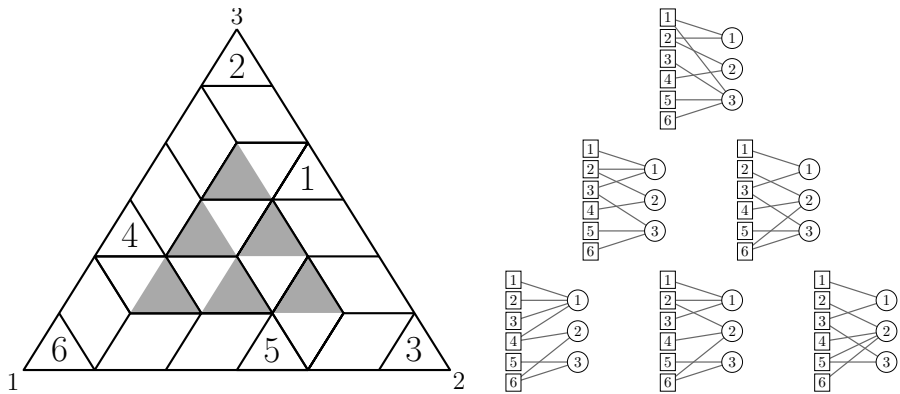

The next lemma together with Lemma~\ref{lem:linkage+and+tree} shows that Definition~\ref{def:linkage+covector} agrees with Definition~\ref{def:linkage+axiom} for matching fields. 

	Fix a linkage $(n,d)$-tope field $\cM$ of type $v$ and thickness $k$.
	Consider a linkage covector $C$ on $\tau \sqcup \rnoset$ with $|\tau| = k+1$. 

\begin{lemma} \label{lem:unique-tope-linkage}
  For each $\rno_i \in \rnoset$, there are exactly $v_i+1$ nodes adjacent to $\rno_i$ in $C$. 
  Furthermore, for each $\lno_j \in \tau$, $C$ contains the unique tope of the tope field with right degree vector $v$ such that $\lno_j$ is isolated.
\end{lemma}
\begin{proof}
  Fix a node $\rno_i \in \rnoset$ and consider two different topes which have only isolated nodes in $\lnoset \setminus \tau$.
  Choose the second such that a neighbour of $\rno_i$ is isolated in the first.
  This shows that $\rno_i$ is adjacent to at least $v_i + 1$ nodes.
  Furthermore, summing up degrees across all nodes shows $\rno_i$ must have exactly $v_i +1$ neighbours, else $C$ is not a tree.
	
  For the second claim, the containment of such a tope is guaranteed by the definition of a linkage covector.
  Now, suppose the tope is not unique; then Lemma~\ref{lem:unique-rdv-topes} implies that $C$ contains a cycle.
\end{proof}

Our next construction starts to connect tope fields of different types.

\begin{lemma} \label{lem:enriched-matching}
  For each $\rno_i \in \rnoset$, the tope linkage covector $C$ contains a unique tope with right degree vector $v + \unit{i}$.
  It is obtained as the union of topes of Lemma~\ref{lem:unique-tope-linkage}.
\end{lemma}
\begin{proof}
  Proposition \ref{prop:unique+contained+tope} gives the first claim.
  For the second claim, consider $\lno_j,\lno_{j'}$ adjacent to $\rno_i$.
  Removing either of these nodes yields a tope of type $v$ contained in the tope of type $v + \unit{i}$.
  By Lemma~\ref{lem:unique-tope-linkage} these topes are unique, therefore we can realise the tope of type $v + \unit{i}$ as the union of them.
\end{proof}

For every $(k+1)$-subset $\tau \subseteq \lnoset$, there is a corresponding linkage covector $C_{\tau}$.
Fix some $\rno_i \in \rnoset$.
By Lemma~\ref{lem:enriched-matching}, there exists a unique tope $G_{\tau}$ of type $v + \unit{i}$ contained in $C_{\tau}$.

\begin{definition}
 The set of the graphs $G_{\tau}$ is the \emph{$i$-amalgamation} of the tope field~$\cM$.
 This is a tope field of type $v + \unit{i}$.
\end{definition}

\begin{example}
  The $(4,2)$-tope field of type $(2,1)$ in Figure~\ref{fig:four+two+tope+field} can be induced from the $(4,2)$-matching field in Figure~\ref{fig:four+two+matching+field} with the construction from Proposition~\ref{prop:unique+contained+tope}.
The topes are obtained from each of the linkage covectors by taking the $1$-amalgamation, as we demonstrate in Figure~\ref{fig:four+two+amalgamation}.
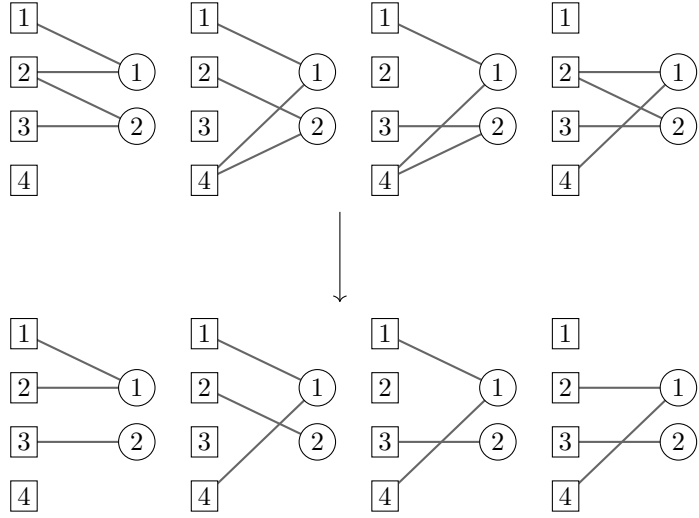
\begin{figure}[htb]
  \begin{center}
  \begin{tikzpicture}[scale=0.6]

\bigraphfourtwocoord{0}{0}{1.2}{0.6}{2.5};
\draw[EdgeStyle] (v1) to (w1);
\draw[EdgeStyle] (v2) to (w1);
\draw[EdgeStyle] (v2) to (w2);
\draw[EdgeStyle] (v3) to (w2);
\bigraphfourtwonodes

\bigraphfourtwocoord{4}{0}{1.2}{0.6}{2.5};
\draw[EdgeStyle] (v1) to (w1);
\draw[EdgeStyle] (v4) to (w1);
\draw[EdgeStyle] (v2) to (w2);
\draw[EdgeStyle] (v4) to (w2);
\bigraphfourtwonodes

\bigraphfourtwocoord{8}{0}{1.2}{0.6}{2.5};
\draw[EdgeStyle] (v1) to (w1);
\draw[EdgeStyle] (v4) to (w1);
\draw[EdgeStyle] (v4) to (w2);
\draw[EdgeStyle] (v3) to (w2);
\bigraphfourtwonodes

\bigraphfourtwocoord{12}{0}{1.2}{0.6}{2.5};
\draw[EdgeStyle] (v2) to (w1);
\draw[EdgeStyle] (v4) to (w1);
\draw[EdgeStyle] (v2) to (w2);
\draw[EdgeStyle] (v3) to (w2);
\bigraphfourtwonodes

\draw[->] (7,-2.5) -- (7,-4.5);

\bigraphfourtwocoord{0}{-7}{1.2}{0.6}{2.5};
\draw[EdgeStyle] (v1) to (w1);
\draw[EdgeStyle] (v2) to (w1);
\draw[EdgeStyle] (v3) to (w2);
\bigraphfourtwonodes

\bigraphfourtwocoord{4}{-7}{1.2}{0.6}{2.5};
\draw[EdgeStyle] (v1) to (w1);
\draw[EdgeStyle] (v4) to (w1);
\draw[EdgeStyle] (v2) to (w2);
\bigraphfourtwonodes

\bigraphfourtwocoord{8}{-7}{1.2}{0.6}{2.5};
\draw[EdgeStyle] (v1) to (w1);
\draw[EdgeStyle] (v4) to (w1);
\draw[EdgeStyle] (v3) to (w2);
\bigraphfourtwonodes

\bigraphfourtwocoord{12}{-7}{1.2}{0.6}{2.5};
\draw[EdgeStyle] (v2) to (w1);
\draw[EdgeStyle] (v4) to (w1);
\draw[EdgeStyle] (v3) to (w2);
\bigraphfourtwonodes

\end{tikzpicture}
  \caption{The $(4,2)$-tope field in Figure \ref{fig:four+two+tope+field} arising as the 1-amalgamation of the $(4,2)$-matching field in Figure \ref{fig:four+two+matching+field}.}
  \label{fig:four+two+amalgamation}
	\end{center}
\end{figure}
\end{example}

In the following, we fix a $(k+2)$-subset $\sigma \subseteq \lnoset$ and denote the tope $G_{\sigma \setminus \lno_s}$ by $G_{s}$.

\begin{lemma} \label{lem:intersection+amalgamation+topes}
  Let $G_j$ and $G_{j'}$ be two topes in the $i$-amalgamation of the $(n,d)$-tope field~$\cM$. Then the neighbourhood of $\rno_i$ differs by at most one element in $G_j$ and $G_{j'}$.
\end{lemma}
\begin{proof}
  The two topes are subgraphs of the linkage covectors $C_j := C_{\sigma\setminus \lno_j}$ and $C_{j'} := C_{\sigma\setminus \lno_{j'}}$ respectively.
  The two linkage covectors $C_j$ and $C_{j'}$ have a tope in common, namely the tope on $\sigma \setminus \{\lno_j,\lno_{j'}\} \sqcup \rnoset$ with right degree vector $v$, which is unique by Lemma~\ref{lem:unique-tope-linkage}.
  As $G_j$ contains all edges of $C_j$ incident with $\rno_i \in \rnoset$, in particular those in the common tope, the neighbourhood of $\rno_i$ differs by at most one element in~$G_j$ and~$G_{j'}$.
\end{proof}

Now we can prove the crucial observation that the linkage property is preserved under amalgamation.

\begin{proposition} \label{prop:i-amalgamation-linkage}
The $i$-amalgamation of a linkage tope field of type $v$ is a linkage tope field of type $v+e_i$.
\end{proposition}
\begin{proof}
  Let $\sigma \subseteq \lnoset$ be a $(k+2)$-subset and define $G$ as the union of the topes $G_{\tau}$ in the $i$-amalgamation where $\tau \subset \sigma$ ranges over the $(k+1)$-subsets.
  We claim that $G$ is a spanning tree with right degree vector $v + \unit{[d]} + \unit{i}$.
  Without loss of generality, we assume $\sigma = \{\lno_1,\dots,\lno_{k+2}\}$ and $i = 1$. 
	At first, we want to show that the degree of $\rno_1$ is $v_1+2$.
	
	Assume $\lno_1$ is adjacent to $\rno_1$ in $G$.
	The tope $G_1$ has $v_1 +1$ nodes adjacent to $\rno_1$, assume these are $\{\lno_2,\dots,\lno_{v_1+2}\}$.
	Hence $\rno_1$ has at least degree $v_1+2$ in $G$.

        Let $G_s$ be a tope containing the edge $(\lno_1,\rno_1)$.
        By Lemma~\ref{lem:intersection+amalgamation+topes}, all other neighbours of $\rno_1$ in $G_s$ form a $v_1$-subset of $\{\lno_2,\dots,\lno_{v_1+2}\}$.
        Define $\lno_t$ to be the unique node such that $\{\lno_1,\lno_2,\dots,\lno_{v_1+2}\} \setminus \{\lno_t\}$ is the neighbourhood of $\rno_1$ in $G_s$.
        Consider a tope $G_j$ for $j \in [v_1+2] \setminus \{t\}$.
        Comparing it with $G_1$ and $G_s$, Lemma~\ref{lem:intersection+amalgamation+topes} shows that the neighbourhood of $\rno_1$ in $G_j$ is $\{\lno_1,\lno_2,\dots,\lno_{v_1+2}\} \setminus \{\lno_j\}$.
        Comparing $G_t$ with such a $G_j$ and $G_1$, the same argument yields that the neighbourhood of $\rno_1$ in $G_t$ is $\{\lno_1,\lno_2,\dots,\lno_{v_1+2}\} \setminus \{\lno_t\}$.
        Analogously, by comparing with $G_1$ and $G_2$, we obtain that the neighbourhood of $\rno_1$ in $G_j$ for $j > v_1+2$ is a subset of $\{\lno_1,\lno_2,\dots,\lno_{v_1+2}\}$.
        Therefore $\rno_1$ has degree $v_1+2$ in $G$.
        
        Lemma~\ref{lem:unique-tope-linkage} implies that $G_1, \dots, G_{v_1+2}$ must agree on $\sigma \setminus \{\lno_1,\dots,\lno_{v_1+2}\} \sqcup \rnoset \setminus \{\rno_1\}$.
        Explicitly, for any two topes $G_s, G_t$, removing the edges $(\lno_t,\rno_1)$ and $(\lno_s,\rno_1)$ from each graph gives two topes of type $v$ on $\sigma \setminus \{\lno_s,\lno_t\}$.
        These topes are equal, as they are the common tope between $C_s$ and $C_t$.

  Fix a node $\rno_i \in \rnoset \setminus \{\rno_1\}$, we want to show that it has degree $v_i +1$.
  Assume that this is not the case, let $\lno_{j_1},\dots,\lno_{j_{v_i}} \in \sigma$ be the nodes adjacent to $\rno_i$ in $G_m$ for all $m \in [v_1+2]$.
  Then there are nodes $\lno_{j_p},\lno_{j_q} \in \sigma \setminus \{\lno_{j_1},\dots,\lno_{j_{v_i}}\}$ and $\lno_p, \lno_q \in \sigma$ such that $\lno_{j_p}$ is adjacent with $\rno_i$ in $G_p$ and $\lno_{j_q}$ is adjacent with $\rno_i$ in $G_q$.
	
  By Lemma \ref{lem:intersection+amalgamation+topes}, $G_p$ and $G_q$ agree on at least $v_1$ edges adjacent to $\rno_1$, assume without loss of generality they are $(\lno_1,\rno_1),\dots,(\lno_{v_1},\rno_1)$.
  Then $G_{v_1+1}$, $G_p$ and $G_q$ all contain these edges.
  Remove the edge $(\lno_1,\rno_1)$ from each graph, the resulting graphs are all topes of type $v$ contained in the linkage covector $C_1$.
  This implies that $\rno_i$ is adjacent to $\lno_{j_1},\dots,\lno_{j_{v_i}},\lno_{j_p},\lno_{j_q}$ in $C_1$, contradicting the property that $\rno_i$ has degree $v_i+1$ in $C_1$.

  Finally, we prove that $G$ is a tree.
  Assuming the opposite, the degree conditions imply that it is not connected.
  Then there are disjoint decompositions $\sigma = J \cup \overline{J}$  and $\rnoset = I \cup \overline{I}$ such that $\rno_1 \in I$ and there are no edges between $J$ and $\overline{I}$ as well as between $I$ and $\overline{J}$.
  For each $\lno_j \in \sigma$, $G$ contains a tope with right degree vector $v + \unit{1}$ such that $\lno_j$ is isolated, therefore we obtain $|J| \geq 2 + \sum_{\rno_i \in I} v_i$ and $|\overline{J}| \geq 1 + \sum_{\rno_i \in \overline{I}} v_i$. This contradicts $|J| + |\overline{J}| = k+2 = \sum_{\rno_i \in I} v_i + \sum_{\rno_i \in \overline{I}} v_i + 2$.
\end{proof}


Proposition \ref{prop:i-amalgamation-linkage} allows us to apply sequences of $i$-amalgamations to obtain a maximal tope from a linkage $(n,d)$-matching field for any right degree vector $(v_1,\dots,v_d)$ such that $\sum v_i = n$. We refer to this construction as \emph{iterated amalgamation}.

\begin{theorem} [Iterated amalgamation] \label{thm:constructed+topes}
  From a linkage matching field, we can construct maximal topes for all positive right degree vectors with sum $n$.
  Each one is the unique tope with a given right degree vector that is compatible with the matching field. 
\end{theorem}
\begin{proof}
  We can apply Proposition~\ref{prop:i-amalgamation-linkage} to obtain linkage tope fields with increasing thickness. By applying an $i$-amalgamation $v_{i}-1$ times iteratively for $i$ from $1$ to $d$, we construct a linkage tope field of type $(v_1,\ldots, v_d)$. Note that a tope field with $\sum_{i = 1}^{d}v_i = n$ contains only a single tope.

  Moreover, there is exactly one tope with right degree vector $(v_1,\ldots, v_d)$ which is compatible with the original linkage matching field.
  Assume, on the contrary, that there are two such topes $T_1$ and $T_2$.
  As a result of Lemma~\ref{lem:enriched-matching}, all the matchings in these topes are matchings of the matching field or submatchings of those.
  By Lemma~\ref{lem:unique-rdv-topes}, these topes differ in a matching.
  This implies that the contained matchings are not weakly compatible which contradicts Proposition~\ref{prop:weak-compatibility}.
  Finally, this implies the uniqueness and compatibility.
\end{proof}

The tope linkage covectors arising from iterated amalgamation do not have to be compatible, see Example~\ref{ex:incompatible+linkage+covectors}.
However, the tope linkage covector is the only tree of that right degree vector which could possibly be compatible with the matching field.
This follows from Proposition~\ref{prop:unique+contained+tope} since it is uniquely determined by the contained maximal topes.
In particular, if that tope linkage covector is not compatible with the matching field then there does not exist a tree with the same right degree vector. 
Hence, such matching fields cannot be obtained via the extraction method from a triangulation of $\Dprod{n-1}{d-1}$.

\begin{example} \label{ex:incompatible+linkage+covectors}
  Consider the $(6,4)$-linkage matching field $\cM$ given by the linkage covectors of matchings in Figure~\ref{fig:non+right+linkage+field}.
  Note that there are multiple pairs of trees that are not compatible. In particular, the fourth tree contains a matching on $\{\lno_5,\lno_6\} \sqcup \{\rno_1,\rno_2\}$ which is incompatible with the matching of $\cM$ on $\{\lno_3,\lno_4,\lno_5,\lno_6\}$ that is contained in the fifth and sixth trees.
	Incompatibility does not affect the amalgamation process and so there is still a unique maximal tope for each $(v_1,\dots,v_d)$ such that $\sum v_i = 6$ and $v_i \geq 1$.
	These are given in Figure~\ref{fig:non+polyhedral+tope+arrangement} by their bijection to the lattice points of $2\Delta_3$.
	
	Figure~\ref{fig:incompatible+linkage+covectors} shows the tope linkage covectors that the maximal topes are derived from.
	They too have a natural lattice point bijection given by their right degree vectors as there is one for each $(v_1,\dots,v_d)$ such that $\sum v_i = 9$ and $v_i \geq 2$.
	However, we note that multiple covectors have the same left degree vectors.
	By \cite[Theorem 12.9]{Postnikov:2009}, these covectors cannot encode the maximal cells of a triangulation of $\Dprod{5}{3}$, and so this matching field is not realisable by a fine mixed subdivision of $6\Delta_3$.
	
	The bijections between various graphs and lattice points will be explored further in Section \ref{sec:chow+covectors}.
\end{example}

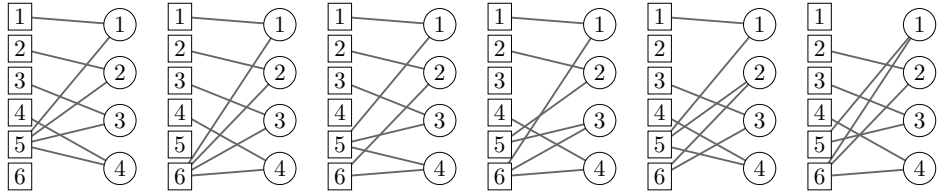
\begin{figure}[htb]
  \begin{center}
	\resizebox{\textwidth}{!}{
  \begin{tikzpicture}[scale = 0.6]
\bigraphsixfourcoord{0}{0}{0.8}{1.2}{2.5};
\draw[EdgeStyle] (v1) to (w1);
\draw[EdgeStyle] (v2) to (w2);
\draw[EdgeStyle] (v3) to (w3);
\draw[EdgeStyle] (v4) to (w4);
\draw[EdgeStyle] (v5) to (w1);
\draw[EdgeStyle] (v5) to (w2);
\draw[EdgeStyle] (v5) to (w3);
\draw[EdgeStyle] (v5) to (w4);
\bigraphsixfournodes

\bigraphsixfourcoord{4}{0}{0.8}{1.2}{2.5};
\draw[EdgeStyle] (v1) to (w1);
\draw[EdgeStyle] (v2) to (w2);
\draw[EdgeStyle] (v3) to (w3);
\draw[EdgeStyle] (v4) to (w4);
\draw[EdgeStyle] (v6) to (w1);
\draw[EdgeStyle] (v6) to (w2);
\draw[EdgeStyle] (v6) to (w3);
\draw[EdgeStyle] (v6) to (w4);
\bigraphsixfournodes

\bigraphsixfourcoord{8}{0}{0.8}{1.2}{2.5};
\draw[EdgeStyle] (v1) to (w1);
\draw[EdgeStyle] (v2) to (w2);
\draw[EdgeStyle] (v3) to (w3);
\draw[EdgeStyle] (v5) to (w4);
\draw[EdgeStyle] (v5) to (w1);
\draw[EdgeStyle] (v6) to (w2);
\draw[EdgeStyle] (v5) to (w3);
\draw[EdgeStyle] (v6) to (w4);
\bigraphsixfournodes

\bigraphsixfourcoord{12}{0}{0.8}{1.2}{2.5};
\draw[EdgeStyle] (v1) to (w1);
\draw[EdgeStyle] (v2) to (w2);
\draw[EdgeStyle] (v6) to (w3);
\draw[EdgeStyle] (v4) to (w4);
\draw[EdgeStyle] (v6) to (w1);
\draw[EdgeStyle] (v5) to (w2);
\draw[EdgeStyle] (v5) to (w3);
\draw[EdgeStyle] (v6) to (w4);
\bigraphsixfournodes

\bigraphsixfourcoord{16}{0}{0.8}{1.2}{2.5};
\draw[EdgeStyle] (v1) to (w1);
\draw[EdgeStyle] (v6) to (w2);
\draw[EdgeStyle] (v3) to (w3);
\draw[EdgeStyle] (v4) to (w4);
\draw[EdgeStyle] (v5) to (w1);
\draw[EdgeStyle] (v5) to (w2);
\draw[EdgeStyle] (v6) to (w3);
\draw[EdgeStyle] (v5) to (w4);
\bigraphsixfournodes

\bigraphsixfourcoord{20}{0}{0.8}{1.2}{2.5};
\draw[EdgeStyle] (v6) to (w1);
\draw[EdgeStyle] (v2) to (w2);
\draw[EdgeStyle] (v3) to (w3);
\draw[EdgeStyle] (v4) to (w4);
\draw[EdgeStyle] (v5) to (w1);
\draw[EdgeStyle] (v6) to (w2);
\draw[EdgeStyle] (v5) to (w3);
\draw[EdgeStyle] (v6) to (w4);
\bigraphsixfournodes
\end{tikzpicture}
	}
  \caption{The linkage covectors of the linkage matching field discussed in Example~\ref{ex:incompatible+linkage+covectors}.}
  \label{fig:non+right+linkage+field}
	\end{center}
\end{figure}

\begin{figure}[htb] 
  \includegraphics[width=0.7\textwidth]{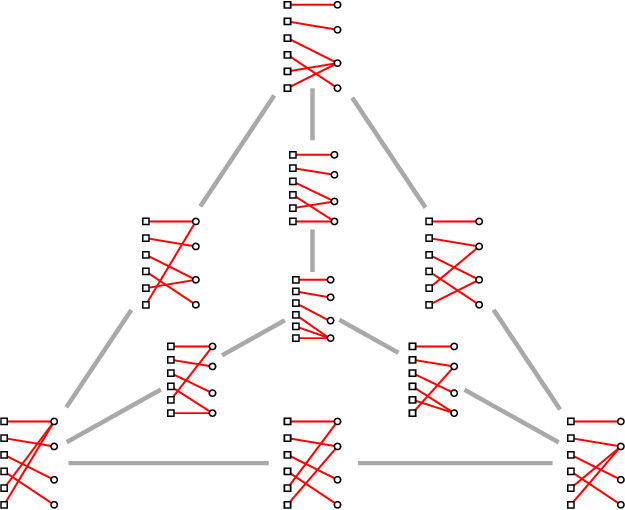}
  \caption{The $(6,4)$-tope arrangement of maximal topes derived from the matching field in Example~\ref{ex:incompatible+linkage+covectors}.
	There is a natural bijection with the lattice points of $2\Delta_3$ via their right degree vector minus $\unit{[4]}$.}
  \label{fig:non+polyhedral+tope+arrangement}
\end{figure}

\begin{figure}[htb] 
  \includegraphics[width=0.6\textwidth]{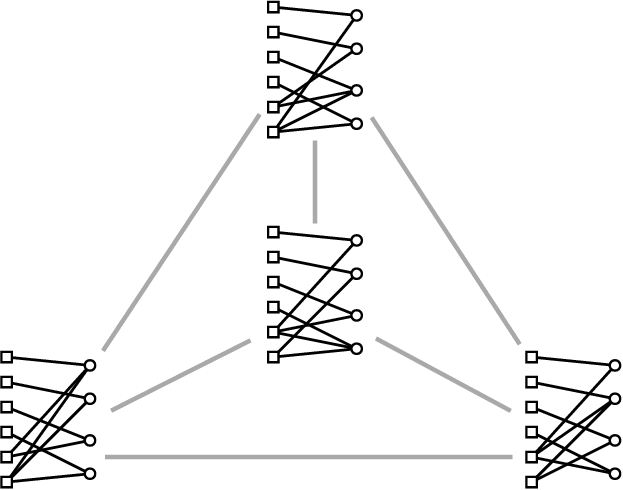}
  \caption{The tope linkage covectors that the maximal topes in Figure \ref{fig:non+polyhedral+tope+arrangement} are derived from.
	There is a natural bijection with the lattice points of $\Delta_3$ via their right degree vector minus $\unit{[4]}$.}
  \label{fig:incompatible+linkage+covectors}
\end{figure}

This warrants the introduction of a new distinct class of matching fields: we say an $(n,d)$-matching field is \emph{polyhedral} if it can be obtained from a triangulation of $\Dprod{n-1}{d-1}$ via the extraction method.
Polyhedral matching fields are a distinct subclass of linkage matching fields that contain coherent matching fields as a distinct subclass. 
Additionally, \cite[Proposition 2.3]{SturmfelsZelevinsky:1993} gives an example of a non-polyhedral $(5,3)$-matching field, the smallest values of $(n,d)$ for which such a matching field exists.

Figure \ref{fig:matching+field+classes} shows the relationship between these classes and triangulations of $\Dprod{n-1}{d-1}$.
Note that there is a variety of equivalent representations: we can replace (regular) triangulations of $\Dprod{n-1}{d-1}$ with either (regular) fine mixed subdivisions of $n\Delta_{d-1}$~\cite{Santos:2005}, (coherent) matching ensembles~\cite{OhYoo-ME:2013}, or generic tropical oriented matroids (generic tropical hyperplane arrangements in the regular case)~\cite{Horn1}.
	
\begin{figure}
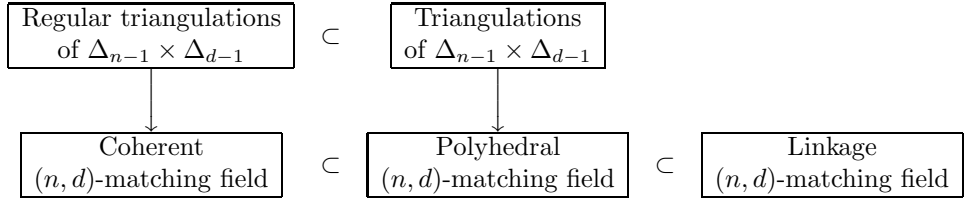

\begin{tabular}{ccccc}
\begin{tabular}{|c|} \hline Regular triangulations \\ of $\Dprod{n-1}{d-1}$ \\ \hline \end{tabular} & $\subset$ & \begin{tabular}{|c|} \hline Triangulations \\ of $\Dprod{n-1}{d-1}$ \\ \hline \end{tabular} & & \\
$\bigg\downarrow$ & & $\bigg\downarrow$ & & \\
\begin{tabular}{|c|} \hline Coherent \\ $(n,d)$-matching field \\ \hline \end{tabular} & $\subset$ & \begin{tabular}{|c|} \hline Polyhedral \\ $(n,d)$-matching field \\ \hline \end{tabular} & $\subset$ & \begin{tabular}{|c|} \hline Linkage \\ $(n,d)$-matching field \\ \hline \end{tabular}
\end{tabular}
\caption{The relationship between triangulations of $\Dprod{n-1}{d-1}$ and different classes of matching fields. The maps shown correspond to the extraction method.}
\label{fig:matching+field+classes}
\end{figure}

Proposition \ref{prop:i-amalgamation-linkage} and Theorem \ref{thm:constructed+topes} have the following interpretation in terms of triangulations of $\Delta_{n-1} \times \Delta_{d-1}$.
Recall that, by the construction in Section \ref{sec:matching+fields+from+triangulations}, every bipartite graph $G$ on $\lnoset \sqcup \rnoset$ can be associated to a subpolytope $P(G)$ of $\Dprod{n-1}{d-1}$.
Given a tope field $\cM = (M_{\sigma})$ of thickness $k$, the corresponding polytope $P(M_{\sigma})$ is a $(k-1)$-dimensional simplex by \cite[Lemma 12.5]{Postnikov:2009}.

Given some $(k+1)$-subset $\tau \subseteq \lnoset$, the polytope $P(C_{\tau})$ is the convex hull of those $P(M_{\sigma})$ where $\sigma \subset \tau$.
The matching field $\cM$ is linkage if and only if each $C_{\tau}$ is a forest, or equivalently each $P(C_{\tau})$ is a simplex (of dimension $k+d-1$).
By Lemma \ref{lem:enriched-matching}, each $P(C_{\tau})$ contains $P(G_{\tau})$ as a face of dimension $k$.
Explicitly, it is the unique simplicial face of $P(C_{\tau})$ that maximises the linear functional $\sum_{j \in \tau} x_{\lno_j,\rno_i}$ and intersects the interior of $P(K_{\tau,\rnoset}) = \Dprod{\tau}{d-1}$.
Furthermore, its existence is only guaranteed if $P(C_{\tau})$ is a simplex.

Proposition \ref{prop:i-amalgamation-linkage} implies that for each $(k+2)$-subset $\rho \subseteq \lnoset$, the convex hull of $P(G_{\tau})$, where $\tau \subset \rho$, is also a simplex.
This allows us to iterate the procedure until we have a set of $(n-1)$-dimensional simplices corresponding to the maximal topes of the matching field.
Furthermore each of these topes are compatible and so these simplices form a simplicial complex $\Sigma(\cM)$.
This is not necessarily true for the simplices $P(C_{\tau})$ as Example~\ref{ex:incompatible+linkage+covectors} demonstrates, and so $\Sigma(\cM)$ may not extend to a full triangulation of $\Dprod{n-1}{d-1}$.

There has been numerous work concerning when a triangulation of the boundary $\Delta_{n-1}^{(m-1)} \times \Delta_{d-1}$ can be extended to a triangulation of $\Dprod{n-1}{d-1}$ \cite{ArdilaCeballos:2013,Santos:2013,CeballosPadrolSarmiento:2015}.
As a corollary to Proposition \ref{prop:i-amalgamation-linkage}, we get a result concerning the complementary problem of extendibility of subcomplexes of the interior of $\Dprod{n-1}{d-1}$ to triangulations.
We call $P(G) \subseteq \Dprod{n-1}{d-1}$ an \emph{interior polytope} if it intersects the interior of $\Dprod{n-1}{d-1}$.
In terms of bipartite graphs, this is equivalent to $G$ having no isolated nodes. 
In particular, as topes have no isolated nodes then $\Sigma(\cM)$ is a simplicial complex of interior polytopes.

\begin{corollary}
There exists an $(n-1)$-dimensional simplicial complex of interior polytopes of $\Dprod{n-1}{d-1}$ that does not extend to a triangulation of $\Dprod{n-1}{d-1}$.
\end{corollary}

This gives a non-extendibility result in the vein of \cite[Theorem 4.3]{CeballosPadrolSarmiento:2015}.
It would be interesting to see if this result was tight, and a similar result to \cite[Theorem 4.2]{CeballosPadrolSarmiento:2015} held for these interior simplicial complexes. 

\subsection{Chow covectors} \label{sec:chow+covectors}

The graphs in the next definition were first introduced in \cite{SturmfelsZelevinsky:1993} in the guise of brackets.
They were used to explicitly describe extremal terms of the Chow form of the variety of complex degenerate $(n \times d)$-matrices, as well as to describe a universal Gr\"{o}bner basis of the ideal generated by the maximal minors of a matrix of indeterminates.
We shall define and consider them in purely combinatorial terms as bipartite graphs.

\begin{definition} \label{def:Chow-covector}
 Let $\cM = (M_{\sigma})$ be a matching field.
For every $(n-d+1)$-subset $\rho \subset \lnoset$ we define the \emph{Chow covector} as the graph of the mapping $\Omega_{\rho} \colon \rho \rightarrow \rnoset$ with
  \[
  \Omega_{\rho}(j) = M_{\bar{\rho}\cup\{j\}}(j) \enspace
  \]
  where $\bar{\rho} = \lnoset \setminus \rho$.
\end{definition}

\begin{remark}
The Chow covectors have a combinatorial characterisation that we shall exploit later.
A graph $G$ on $\lnoset \sqcup \rnoset$ is \emph{transversal} to a matching field $\cM$ if $G \cap M \neq \emptyset$ for all $M \in \cM$.
Bernstein and Zelevinsky \cite[Theorem 1]{BernsteinZelevinsky:1993} showed that the Chow covectors are the minimal transversals to a matching field.
\end{remark}

\begin{figure}[htb]
\includegraphics[width=\columnwidth]{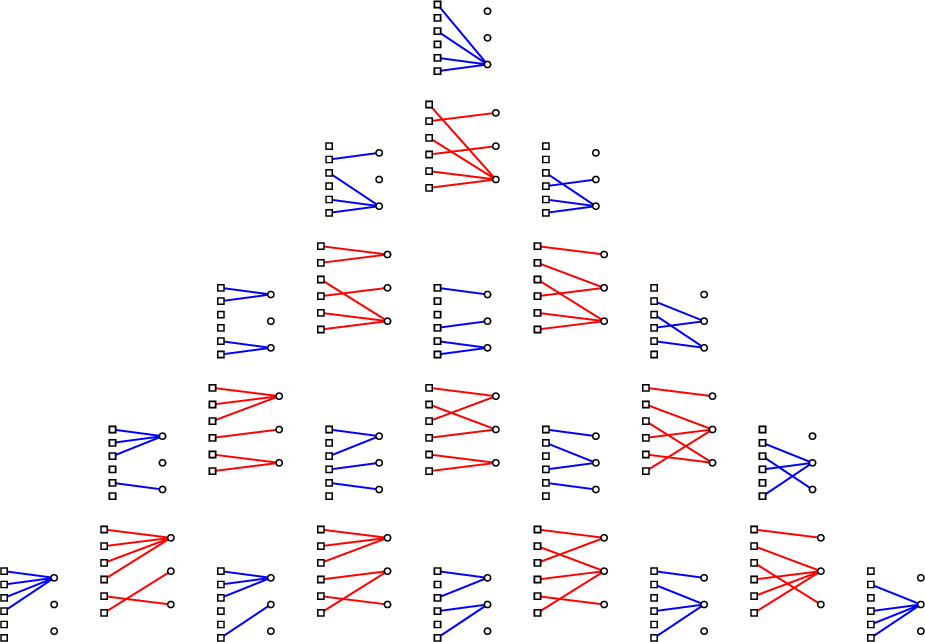}%
\caption{The topes and Chow covectors corresponding to the $(6,3)$-matching field from Figure \ref{fig:nonregular+subdivision}, presented via their bijection to $\simplat{3}{3}$ and $\simplat{4}{3}$ respectively. The topes are coloured red and the Chow covectors are coloured blue.}%
\label{fig:chow+lattice}%
\end{figure}

Again, fix a linkage $(n,d)$-matching field $\cM$.
We will derive the Chow covectors of a linkage matching field from its topes.
To do this we need a statement similar to Lemma~\ref{lem:abstract-containment-sector}.

\begin{lemma} \label{lem:neighbour-topes}
  Let $T_1$ and $T_2$ be two compatible topes with right degree vectors $v$ and $v + \unit{p} - \unit{q}$ (where $v_q \geq 2$). 
  The edges incident with $\rno_q$ in $T_2$ are all contained in $T_1$.
\end{lemma}

\begin{proof}
  We define an auxiliary directed graph $H$ on $\lnoset \sqcup \rnoset$. The set of arcs is given by the edges of $T_1 \setminus T_2$ directed from $\lnoset$ to $\rnoset$ and of $T_2 \setminus T_1$ directed from $\rnoset$ to $\lnoset$.
  Observe the following degree properties of $H$.
	Every node in $\lnoset$ has in-degree equal to out-degree equal to either $1$ or $0$.
	Every node in $\rnoset \setminus \{\rno_p, \rno_q\}$ has in-degree equal to out-degree.

  Now, assume that the claim does not hold, which means that $\rno_q$ has in-degree bigger or equal to $2$.
	As the out-degree of $\rno_q$ is just $1$ less than the in-degree, it has out-degree at least $1$.
  Consequently, since the sum of the right degree of $T_1$ and $T_2$ are the same, the out-degree of $\rno_p$ has to be bigger or equal to $2$.
  Hence, all non-isolated nodes have out-degree at least $1$ which means that $H$ contains a directed cycle.
  This is necessarily alternating between edges of $T_1$ and $T_2$.
  However, this contradicts the compatibility of $T_1$ and $T_2$, concluding the claim.
\end{proof}

\begin{proposition} \label{prop:construct-Chow}
Given a linkage matching field $\cM$, there is a unique Chow covector with right degree vector $v$ associated to $\cM$. It can be constructed from the intersection of the maximal topes of $\cM$ with right degree vector $v + \unit{[d] \setminus \{i\}}$ for $i \in [d]$.
\end{proposition}
\begin{proof}
  Let $v$ be a vector of non-negative integers with coordinate sum $n - d +1$. 
  We fix a node $\rno_i \in \rnoset$ and consider the graph obtained by intersecting the topes with right degree vectors $v + \unit{[d] \setminus \{i\}}$.
  If $v_i = 0$, there is no tope with right degree vector $v + \unit{[d] \setminus \{i\}}$.
  Instead, we delete any edge adjacent to $\rno_i$ from the graph.
  Denote this graph by $G$.
  By Lemma~\ref{lem:neighbour-topes}, the topes which we intersect agree on the edges of the tope where a given node has degree $v_i$.
	For an example, see the red graphs surrounding a blue graph in Figure~\ref{fig:chow+lattice}. 
  Hence, $G$ has $n-d+1$ edges, each node $\rno_i \in \rnoset$ has degree $v_i$ and there are $n - d +1$ nodes in $\lnoset$ with degree $1$, all others having degree~$0$.

  Let $\rho \subset \lnoset$ be the set of left nodes with degree 1 in~$G$.
	We claim that $G$ is the Chow covector $\Omega_{\rho}$. 
  Fix an $\ell_j \in \rho$ adjacent to the node $\rno_i$. The maximal tope with right degree vector $v + \unit{[d] \setminus \{i\}}$ exists, and contains a matching on $\left(\lnoset\setminus\rho\right) \cup \{\ell_j\}$.
	Explicitly, it is precisely the union of the edges not contained in $G$ with $(\lno_j,\rno_i)$.
  This has to be a matching of the matching field by Lemma~\ref{lem:enriched-matching}.
  Hence, $\Omega_{\rho}$ is a subgraph of $G$ and, hence, equal because of the given degrees of $G$.
\end{proof}

Consider the dilated simplex $n\Delta_{d-1}$ with its canonical embedding into $\RR^d$.
We denote the set of integer lattice points of $n\Delta_{d-1}$ by $\simplat{n}{d}$.
Observe that Theorem \ref{thm:constructed+topes} gives a bijection between the set of maximal topes of a matching field and $\simplat{n-d}{d}$ by mapping a tope with right degree vector $v$ to the lattice point $v - \unit{[d]}$.
Sturmfels and Zelevinsky~\cite[Conjecture 6.10]{SturmfelsZelevinsky:1993} conjectured a similar bijection for Chow covectors.
The construction in Proposition \ref{prop:construct-Chow} allows us to complete the proof of this conjecture.

\begin{theorem} \label{thm:bijection-chow-covector-lattice-points}
The map from the Chow covectors of a linkage $(n,d)$-matching field to the lattice points of $(n-d+1)\Delta_{d-1}$ is a bijection.
\end{theorem}

\begin{proof}
For each element of $\simplat{n-d+1}{d}$, Proposition~\ref{prop:construct-Chow} gives us an explicit construction of the Chow covector with that right degree vector, therefore this map is surjective.
Furthermore, the set of Chow covectors and $\simplat{n-d+1}{d}$ both have cardinality $\binom{n}{n-d+1}$, therefore this map is a bijection.
\end{proof}

Observe that the construction in Proposition \ref{prop:construct-Chow} can be inverted.

\begin{corollary} \label{cor:chow+determines+field}
A linkage matching field is uniquely determined by its set of Chow covectors.
\end{corollary}
\begin{proof}
  Given the set of Chow covectors of a linkage matching field, we can recover the tope with right degree vector $v$ by taking the union of the Chow covectors with right degree vector $v - \unit{[d]\setminus\{i\}}$.
  By the last part of the construction in Proposition \ref{prop:construct-Chow}, we know that the topes contain a matching on all $d$-subsets of $L$.
	As the topes are all compatible this implies that those matchings form exactly the matching field.
\end{proof}

The previous corollary has an alternative non-constructive proof by considering the linkage matching field $\cM$ as a hypergraph with vertices and hyperedges
\begin{align*}
V(\cM) &= \SetOf{(\lno_j,\rno_i)}{\lno_j \in \lnoset \ , \ \rno_i \in \rnoset} \\
E(\cM) &= \SetOf{M_{\sigma}}{\sigma \in \binom{[n]}{d}} \enspace .
\end{align*}
The Chow covectors are the \emph{blocker} $b(\cM)$ of $\cM$, the hypergraph whose edges are minimal transversals to $\cM$.
As $b(b(\cM)) = \cM$ for any hypergraph with incomparable edges \cite[Theorem 77.1]{Schrijver:2003}, this immediately implies Corollary \ref{cor:chow+determines+field}.

\begin{example} \label{ex:blue-red-Chow-topes}
  Figure \ref{fig:chow+lattice} shows the topes, coloured red, corresponding to the $(6,3)$-matching field derived from the non-regular triangulation illustrated in Figure \ref{fig:nonregular+subdivision} (see Subsection~\ref{sec:matching+fields+from+triangulations}).
  As there is a unique tope for every possible right degree vector, these form a bijection with $\simplat{3}{3}$, the lattice points of the simplex $3\Delta_2$.
  The Chow covectors, coloured blue, of this matching field can be recovered from the topes via the procedure described in the proof of Proposition~\ref{prop:construct-Chow}.
  There is precisely one for every lattice point in $\simplat{4}{3}$ as encoded by their right degree vectors.
  The construction is closely related to the mixed subdivision representation of a triangulation which is given by the Cayley Trick \cite{Santos:2005}.
\end{example}

As we will see in Section~\ref{sec:pairs+of+lattice+points}, the trees encoding a triangulation of $\Delta_{n-1} \times \Delta_{d-1}$ are completely determined by the pairs of lattice points given by their left and right degree vectors.
We show that this also holds for the bijection associated with the Chow covectors.
This is the generalisation of \cite[{Conjecture 6.8 b)}]{SturmfelsZelevinsky:1993} discussed below their claim.
For this, we need two lemmas.

\begin{lemma} \label{lem:Chow-sub-matching-field}
  Given the set of Chow covectors of a linkage matching field on $L \sqcup R$, the set
  \[
	\SetOf{\Omega_{\rho}^{(j)}}{\Omega_{\rho}^{(j)} \text{ restriction of $\Omega_{\rho}$ to } \left(\lnoset \setminus \{\lno_j\}\right) \sqcup \rnoset, \lno_j \in \rho}
  \]
  is the set of Chow covectors of the induced submatching field on $\left(L \setminus \{\lno_j\}\right) \sqcup R$.
\end{lemma}
\begin{proof}
	A matching $\mu$ in the induced submatching field is a matching in the original matching field that does not contain any edges adjacent to $\lno_j$.
	By transversality from~\cite{BernsteinZelevinsky:1993}, $\mu \cap \Omega_{\rho}$ is non-empty.
	As $\mu$ has no edges adjacent to $\lno_j$ this implies $\mu \cap \Omega_{\rho}^{(j)}$ is non-empty as well and, hence, $\Omega_{\rho}^{(j)}$ is a transversal on the set $\rho \setminus \{\lno_j\}$.
	As it contains the same number of edges as a Chow covector, by minimality it must be equal to the Chow covector $\Omega_{\rho \setminus \{\lno_j\}}$.
	Iterating over all $\rho$, we obtain all Chow covectors on $(\lnoset \setminus \{\lno_j\}) \sqcup \rnoset$.
\end{proof}

Finally, we need an analogous lemma to Lemmas~\ref{lem:neighbour-topes} and \ref{lem:abstract-containment-sector}.

\begin{lemma} \label{lem:neighbour-Chow}
  Let $T_1$ and $T_2$ be two Chow covectors with right degree vectors $v$ and $v + \unit{p} - \unit{q}$ (where $v_q \geq 1$) constructed from the same linkage matching field.
  The edges incident with $\rno_q$ in $T_2$ are all contained in $T_1$.
	Furthermore, if $\lno_s$ has degree 1 in both covectors, $T_1$ and $T_2$ agree on the edge adjacent to $\lno_s$.
\end{lemma}
\begin{proof}
  The first claim follows from Lemma~\ref{lem:neighbour-topes} with the construction in Proposition~\ref{prop:construct-Chow}.
For the second, consider the tope with right degree vector $v + \unit{[d] \setminus q}$.
By Proposition \ref{prop:construct-Chow}, both Chow covectors are contained in this tope and so agree with it on the edge adjacent to $\lno_s$.
\end{proof}

Let $\mathcal{M}$ be a linkage matching field.
Define $\varphi_{\mathcal{M}} \colon \binom{[n]}{n-d+1} \rightarrow \simplat{n-d+1}{d}$ to be the bijection mapping $\rho \in \binom{[n]}{n-d+1}$ to the right degree vector of $\Omega_{\rho}$.
We present a proof of Conjecture~6.8~b) of \cite{SturmfelsZelevinsky:1993} that this map uniquely determines the linkage matching field.
To do so, we introduce a combinatorial analogue to sectors of tropical hyperplanes, as discussed in Example \ref{ex:hyperplane+sector}.
	\begin{definition} \label{def:combinatorial+sector}
	Let $\cG$ be a collection of compatible bipartite graphs on $\lnoset \sqcup \rnoset$ with a bijective map $\cG \rightarrow \simplat{k}{d}$ determined by right degree vectors for some $k$.
	The \emph{(open) combinatorial sector} $\cS_j^{(i)}$ is defined as follows:
	\[
	\cS_j^{(i)} = \SetOf{G \in \cG}{(\lno_j,\rno_i) \in G, \lno_j \text{ has degree } 1}
	\]
	\end{definition}
	We have seen multiple classes of bipartite graphs with a bijection to lattice points of dilated simplices, namely trees, topes and Chow covectors, and so we can define combinatorial sectors for any of them.
	Furthermore, all have similar local properties (see Lemmas \ref{lem:abstract-containment-sector},\ref{lem:neighbour-topes} and \ref{lem:neighbour-Chow}) that give a lot of structure to the combinatorial sectors.
	
	In particular, let $\cG = \SetOf{\Omega_{\rho}}{\rho \in \binom{[n]}{n-d+1}}$ where $k=n-d+1$ and consider the combinatorial sectors of the set of Chow covectors.
	For example, Figure \ref{fig:combinatorial+sector} shows the Chow
covectors from Example \ref{ex:blue-red-Chow-topes} with the combinatorial sectors $\cS_1^{(1)}$ and $\cS_1^{(3)}$ highlighted.

\begin{figure}%
\includegraphics[width=0.7\textwidth]{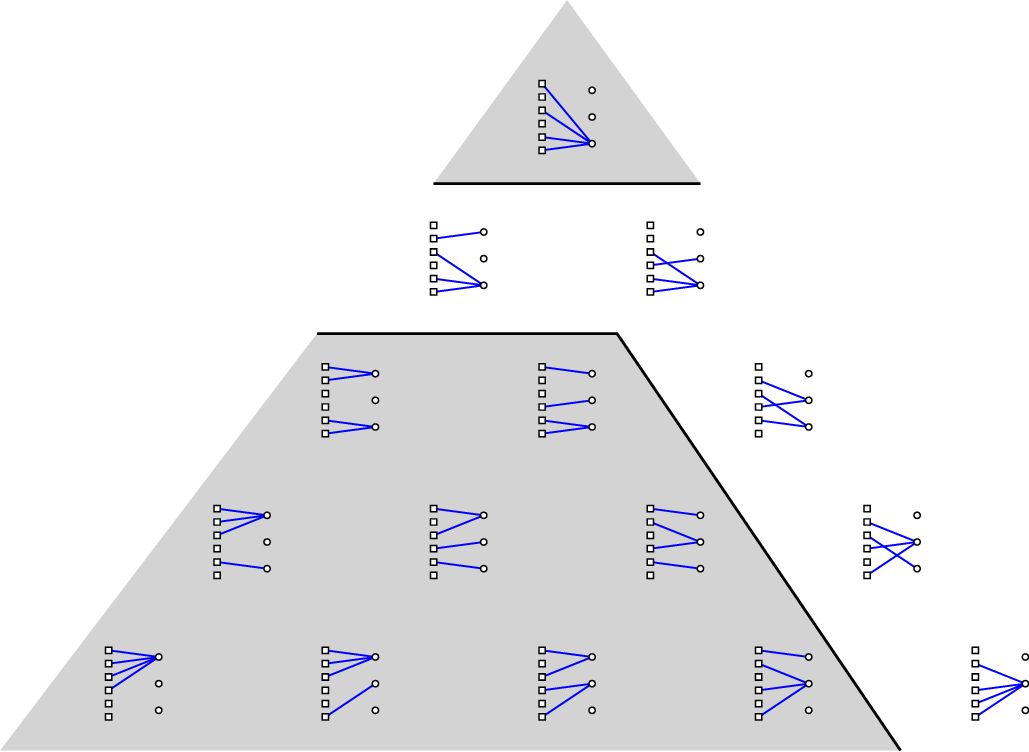}
\caption{Decomposition of the Chow covectors from Example~\ref{ex:blue-red-Chow-topes} depending on the neighbouring vertex of $\lno_1$; it is either adjacent to $\rno_3$, isolated or adjacent to $\rno_1$.
The grey regions are the combinatorial sectors $\cS_1^{(1)}$ and $\cS_1^{(3)}$.
Note that $\cS_1^{(2)}$ is empty as the edge $(\lno_1,\rno_2)$ appears in no Chow covector.}%
\label{fig:combinatorial+sector}%
\end{figure}

\begin{theorem} \label{thm:lattice-points-Chow}
  A linkage matching field can be uniquely determined by its map~$\varphi$.
\end{theorem}
\begin{proof}
  We claim that we can reconstruct the Chow covectors from their left and right degree vector pairs.
	As $\varphi_{\cM}$ is given by these degree vectors, the theorem follows from this claim and Corollary \ref{cor:chow+determines+field}.
	We proceed by induction on $n \geq d$.
	For $n = d$, the matching field consists of one matching and the $n$ Chow covectors are just the edges of the matching.
	The degree vector pair of an edge uniquely determines it, which implies the claim for this case.

  Assume that the claim is true for all linkage $(k,d)$-matching fields with $d \leq k < n$ and let $\mathcal{U}$ be the set of degree vector pairs of the Chow covectors for a linkage $(n,d)$-matching field.
	We get a non-disjoint decomposition
\[
\bigcup_{j \in [n]} \mathcal{L}_j = \mathcal{U} \quad \mbox{for} \quad \mathcal{L}_j = \{(u,v) \mid u_j = 1\} \enspace .
\]
Now fix a $j \in [n]$. There is a partition of $\mathcal{L}_j$ in the sets
\[
\mathcal{L}_j^{(i)} = \{ (u,v) \mid (\lno_j,\rno_i) \in \Omega_{\rho},\;\Omega_{\rho}~\mbox{has degree vector}~(u,v)~\mbox{with}~u_j=1\} \enspace .
\]
Note that $\mathcal{L}_j^{(i)}$ is the image of $\cS_j^{(i)}$ under the map that sends a Chow covector to its degree vector pair.

From $\mathcal{L}_j^{(i)}$ we can construct the set $\overline{\mathcal{L}_j^{(i)}}$ by removing the $j$th entry of the first component and decreasing the $i$th entry of the second component for all the pairs in $\mathcal{L}_j^{(i)}$.
This corresponds to removing the edge $(\lno_j,\rno_i)$.
The resulting set
\[
\overline{\mathcal{L}_j} = \bigcup_{i \in [d]} \overline{\mathcal{L}_j^{(i)}}
\]
is the set of degree pairs of the Chow covectors of the submatching field on $\left(L \setminus \{\lno_j\}\right) \sqcup R$, by Lemma~\ref{lem:Chow-sub-matching-field}.

Here, we can apply induction and deduce that $\overline{\mathcal{L}_j}$ uniquely defines the Chow covectors with the contained degree vectors. From the partition into the $\overline{\mathcal{L}_j^{(i)}}$ we can recover to which node $\lno_j$ is incident in the original Chow covector. Therefore, we can construct all Chow covectors for which $\lno_j$ has degree $1$. Ranging over all $j \in [n]$, we get all the Chow covectors.

\smallskip

It remains to show how to construct the set $\mathcal{L}_j^{(i)}$ for each $\rno_i \in R$, which we now demonstrate.
Assume without loss of generality that $i = 1$ and apply Algorithm~\ref{algo:derive+Chow+sectors}.
\begin{algorithm}[htbp]
  \caption{Construct the degree pairs of a combinatorial sector of Chow covectors}
  \label{algo:derive+Chow+sectors}
  \begin{algorithmic}[1]
  	\Statex \textbf{Input:} $\mathcal{U}$, the set of degree vector pairs of Chow covectors of a linkage $(n,d)$-matching field.
  	\Statex \textbf{Output:} $\mathcal{L}_j^{(1)}$, the set of degree vector pairs with $u_j = 1$ whose Chow covector contains $(\lno_j,\rno_1)$
    \If{$u_j = 1$ for $(u,v)$ with $v_1 = n-d+1$}
    \State $\mathcal{K}_j \gets \{(u,v)\}$ \label{line:initial+K+Chow}
    \Else
    \State \Return $\emptyset$
    \EndIf
    \State $h \gets n-d$
    \While{$h \geq 0$}
    \ForAll{$(u,v) \in \mathcal{L}_j~\mbox{with}~v_1=h$}
    \If{$\exists k \in [d] \colon v_k > 1 \colon \exists\,w^{(k)} \colon (w^{(k)},v+e_1-e_k) \in \mathcal{K}_j$} \label{line:ancestor+Chow}
    \State $\mathcal{K}_j \gets \mathcal{K}_j \cup (u,v)$
    \EndIf
    \State $h \gets h-1$
    \EndFor
    \EndWhile 
    \State \Return $\mathcal{K}_j$
\end{algorithmic}
\end{algorithm}

\textbf{Claim:} $\mathcal{K}_j = \mathcal{L}_j^{(1)}$.

\textbf{Proof by induction} 
There is a unique Chow covector $\Omega_{\rho_0}$ with the right degree vector $(n-d+1)\unit{1}$. 
If $u_j = 1$ then $\lno_j$ is adjacent to $\rno_1$ because of the structure of the right degree vector.
Line~\ref{line:initial+K+Chow} in the algorithm guarantees that $\Omega_{\rho_0}$ is in $\mathcal{K}_j$. Furthermore, the edge $(\lno_j,\rno_1)$ shows that it is also contained in $\mathcal{L}_j^{(1)}$.

Now, assume that $\mathcal{K}_j$ and $\mathcal{L}_j^{(1)}$ agree in all elements whose first entry of the second component is $h+1 \leq n-d+1$. 

	Let $(u,v) \in \mathcal{L}_j$ such that $v_1 = h$ and $(w,v+e_1-e_k) \in \mathcal{K}_j$ an element fulfilling the condition in Line~\ref{line:ancestor}.
	These two vectors are the right degree vectors of two Chow covectors $\Omega_{\rho_1}$ and $\Omega_{\rho_2}$.
	Note that $\mathcal{K}_j \subseteq \mathcal{L}_j$. As, by the induction hypothesis, $(\lno_j,\rno_1)$ is an edge of $\Omega_{\rho_2}$ we can deduce with Lemma~\ref{lem:neighbour-Chow} that this is also an edge of $\Omega_{\rho_1}$.
	Hence, $(u,v)$ is an element of $\mathcal{L}_j^{(1)}$.

	Conversely, let $(u,v) \in \mathcal{L}_j^{(1)}$ be with $v_1 = h$.
Also by Lemma~\ref{lem:neighbour-Chow}, there is a $k \in [d]$ and a $w = \unit{\rho}$ for some $\rho$ such that in the Chow covector with degree pair $(w,v+e_1-e_k)$ the node $\lno_j$ is a leaf and it is adjacent to $\rno_1$.
The induction hypothesis implies that $(w,v+e_1-e_k) \in \mathcal{K}_j$.
Now, Line~\ref{line:ancestor} shows that also $(u,v)$ is an element of $\mathcal{K}_j$.
\end{proof}

\subsection{A cryptomorphic description} \label{sec:cryptomorphism}

In Section \ref{sec:tope+from+matching}, we saw how one can construct a set of compatible topes in bijection with $\simplat{n-d}{d}$ from a linkage $(n,d)$-matching field. A slight generalisation of the proof of Proposition~\ref{prop:construct-Chow} and Corollary~\ref{cor:chow+determines+field} leads to a cryptomorphic description of linkage matching fields in terms of topes. 

\begin{definition}
An \emph{$(n,d)$-tope arrangement} is a set of compatible topes in bijection with $\simplat{n-d}{d}$ via the map that sends a tope with right degree vector $v$ to $v - \unit{[d]}$.
\end{definition}

\begin{example}
  Figure~\ref{fig:non+polyhedral+tope+arrangement} shows a $(6,4)$-tope arrangement which cannot arise from a triangulation of $\Delta_5 \times \Delta_3$.
\end{example}

For the construction of the Chow covectors, one only needs Lemma~\ref{lem:neighbour-topes}. Analogously to Corollary~\ref{cor:chow+determines+field} we get a matching for each $d$-subset.
Combining this with Theorem~\ref{thm:constructed+topes} yields the following.

\begin{theorem} \label{thm:cryptomorphism+linkage+mf}
Linkage $(n,d)$-matching fields and $(n,d)$-tope arrangements are cryptomorphic.
Explicitly:
\begin{itemize}
\item the maximal topes of a linkage matching field form a tope arrangement,
\item the set of all maximal matchings in a tope arrangement form a linkage matching field.
\end{itemize}
\end{theorem}

\begin{proof}
The first statement is shown in Theorem \ref{thm:constructed+topes}, remains to show the second statement.
  As tope arrangements satisfy the conditions of Lemma \ref{lem:neighbour-topes} and are in bijection with $\simplat{n-d}{d}$, we can construct the graphs $\Omega_{\rho}$ for all $\rho \in \binom{[n]}{n-d+1}$ via Proposition \ref{prop:construct-Chow}.
  The Chow covector with left degree vector $\rho$ gives rise to the matchings on $[n]\setminus \rho \cup \{j\}$ for all $j \in \rho$.
  As the topes are compatible, each $\rho$ occurs exactly once. 
  
These graphs have the same properties as Chow covectors, in particular that they yield the existence of a perfect matching for every $d$-subset $\sigma \subset \lnoset$ contained in some tope in the tope arrangement in the same way as the construction at the end of the proof of Proposition~\ref{prop:construct-Chow}.
Note that these matchings are unique as the topes are compatible, therefore the tope arrangement induces a matching field.
It remains to show that the matching field is linkage.

Let $\sigma,\sigma' \subset \lnoset$ be distinct $d$-subsets and $\lno_{j'} \in \sigma' \setminus \sigma$.
Consider a tope $T$ that contains the matching on $\sigma$ and consider the node $\rno_i$ adjacent to $\lno_{j'}$.
There exists some node $\lno_j \in \sigma$ adjacent to $\rno_i$ in the matching on $\sigma$, therefore the matching on $\sigma \setminus \{\lno_{j}\} \cup \{\lno_{j'}\}$ agrees with the matching on $\sigma$ outside of $\rno_i$.
This is equivalent to the linkage axiom by Lemma \ref{lem:subset+linkage}.
\end{proof}

\subsection{Comparison with trianguloids}
Trianguloids are a combinatorial object introduced recently in \cite{GalashinNenashevPostnikov2018} to study triangulations of products of simplices.
Tope arrangements have similar structural properties to trianguloids.
To effectively demonstrate this, we introduce \emph{extended $(n,d)$-tope arrangements}, a set of compatible topes in bijection with $\simplat{n}{d}$ via the map sending a tope to its right degree vector.
Note that this requires us to relax our definition of tope to allow for isolated right nodes.
Observe that we can recover a tope arrangement from an extended tope arrangement by removing any topes with isolated nodes.

To allow for more direct comparison, we introduce the trianguloid axioms in the language of bipartite graphs.
\begin{definition}
Let $\cG = \SetOf{G_a}{a \in \simplat{n}{d}}$ be a collection of bipartite graphs on $\lnoset \sqcup \rnoset$ in bijection with the lattice points $\simplat{n}{d}$.
Let $\cN_a(v)$ denote the neighbourhood of $v$ in $G_a$.
The collection $\cG$ is a \emph{pre-trianguloid} if it satisfies the following axioms:
\begin{enumerate}[label=(T\arabic*)]
\item the graph $G_a$ has right degree vector $a$, \label{ax:rdv}
\item each graph has no isolated left nodes, \label{ax:ldv}
\item for $a, a' \in \simplat{n}{d}$ where $a' = a + \unit{p} - \unit{q}$, we have $\cN_a(\rno_p) \subset \cN_{a'}(\rno_p)$. \label{ax:sector}
\end{enumerate}
$\cG$ is a \emph{trianguloid} if it is a pre-trianguloid that satisfies the following \emph{hexagon axiom}:
\begin{enumerate}[resume,label=(T\arabic*)]
\item \label{ax:hexagon} Let $c \in \simplat{n-2}{d}$ and $i,j,k \in [d]$ be distinct such that $\cN_{c+\unit{i}+\unit{j}}(\rno_j) \neq \cN_{c+\unit{j}+\unit{k}}(\rno_j)$.
Then we have 
\[
\cN_{c+\unit{i}+\unit{j}}(\rno_i) = \cN_{c+\unit{i}+\unit{k}}(\rno_i) \ , \ \cN_{c+\unit{i}+\unit{k}}(\rno_k) = \cN_{c+\unit{j}+\unit{k}}(\rno_k) \enspace .
\]
\end{enumerate}
\end{definition}
\begin{figure}
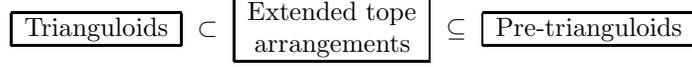

\[
\begin{tabular}{|c|}
\hline
Trianguloids \\
\hline
\end{tabular}
\ \subset \
\begin{tabular}{|c|}
\hline
Extended tope \\
arrangements \\
\hline
\end{tabular}
\ \subseteq \
\begin{tabular}{|c|}
\hline
Pre-trianguloids \\
\hline
\end{tabular}
\]
\caption{The inclusion relationship between trianguloids, pre-trianguloids and extended tope arrangements.}
\label{fig:trianguloid+relationship}
\end{figure}
Extended tope arrangements and (pre-)trianguloids are structurally similar objects, as the following proposition and Figure~\ref{fig:trianguloid+relationship} demonstrate.

\begin{proposition}
Every extended tope arrangement gives rise to a pre-trianguloid.
Trianguloids are a strict subclass of extended tope arrangements. 
\end{proposition}
\begin{proof}
By comparing degrees, axioms~\ref{ax:rdv} and~\ref{ax:ldv} ensure that each graph is a tope, albeit with possibly isolated right nodes.
Axiom~\ref{ax:sector} comprises the combinatorial sector condition which we exhibit for general linkage matching fields in Lemma~\ref{lem:neighbour-topes}.
This lemma directly generalises to extended $(n,d)$-tope arrangements, as isolated right nodes add no additional restrictions.
As the graphs of a (pre-)trianguloid demand no compatibility assumptions, we get the immediate relation that every extended tope arrangement is a pre-trianguloid.

Extended tope arrangements are not required to satisfy the hexagon axiom~\ref{ax:hexagon} as Example~\ref{ex:hexagon+counterexample} and Figure~\ref{fig:extended+tope} demonstrate. 
As a result, extended tope arrangements are a strictly more general class of objects than trianguloids.
\end{proof}

We note that it may be the case that pairwise compatibility is automatically satisfied by pre-trianguloids, showing that global compatibility and the local combinatorial sector conditions are in fact equivalent.
\begin{question}
Are pre-trianguloids and extended tope arrangements cryptomorphic objects?
\end{question}

\begin{figure}
\includegraphics{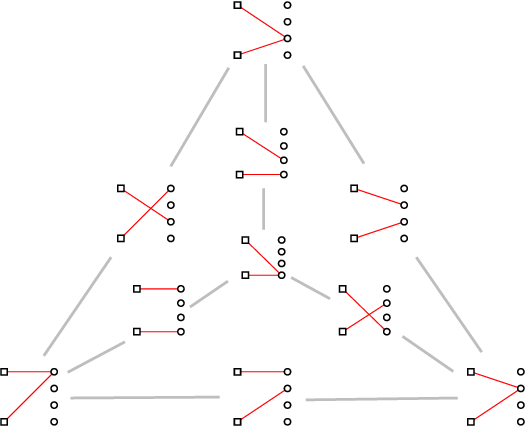}
\caption{An extended $(2,4)$-tope arrangement that does not satisfy the hexagon axiom.}
\label{fig:extended+tope}
\end{figure}
\begin{example} \label{ex:hexagon+counterexample}
Consider the extended $(2,4)$-tope arrangement shown in Figure \ref{fig:extended+tope}.
The topes with right degree vectors $(1,1,0,0), (1,0,1,0), (0,1,1,0)$ all have different neighbourhoods to $\rno_1,\rno_2,\rno_3$.
This violates the hexagon axiom in the case where $c = (0,0,0,0)$ and $i,j,k = \{1,2,3\}$, and is therefore not a trianguloid.

We can see the polyhedral intuition for this via matching field completion, discussed further in Section \ref{sec:matching+stacks+transversal+matroids}.
Observe that by adding four `dummy nodes' $\lno_1,\dots,\lno_4$ to $\lnoset$ and adding the matching consisting of edges $\SetOf{(\lno_i,\rno_i)}{i \in [4]}$ to each tope, we obtain the $(6,4)$-tope arrangement shown in Figure \ref{fig:non+polyhedral+tope+arrangement}.
Recall that this tope arrangement is derived from a linkage matching field cannot be extracted from a triangulation of $\Dprod{5}{3}$.
As trianguloids encode triangulations of products of simplices, it should be unsurprising that this extended $(2,4)$-tope arrangement does not satisfy all the axioms of trianguloids.
\end{example}

The question of realisability of linkage matching fields is addressed further in Section \ref{sec:matching+stacks+transversal+matroids}.
In particular, better understanding the relationships between these objects and of the hexagon axiom may shed light on the problem of extendibility addressed there.

\subsection{Matching field polytopes and the flip graph} \label{sec:flip+graph}

The notion of a \emph{matching field polytope} first occurs in \cite{MohammadiShaw:2018}.
It is the convex hull of the characteristic vectors of the matchings of an $(n,d)$-matching field in $\RR^{n \times d}$.

This is a natural analogue of the matroid polytope, as in some sense matching fields play the role of a matroid for tropical linear algebra.
However, unlike matroid polytopes, their vertex-edge graph is not the flip graph of the matchings as we demonstrate in Example~\ref{ex:flip+graph}.

Here, the nodes of the \emph{flip graph} are the matchings and two matchings are adjacent if and only if they differ in one edge.

\begin{example} \label{ex:flip+graph}
We consider a $(5,3)$-linkage matching field as shown in Figure~\ref{fig:flip+graph}.
The matchings are encoded by words of length three where the matching contains $(\lno_j,\rno_i)$ if $j$ is in the $i$th position.
Two matchings differ by a flip if the words differ in precisely one position.
We note that the flip graph can be decomposed into the linkage trees for each $4$-subset of the set $[5]$, each one represented by a different colour.
\begin{figure}%
\newcount\mycount
\begin{tikzpicture}
\foreach \x in {1,...,10}{%
    \pgfmathparse{(\x-1)*36}
    \coordinate (v\x) at (\pgfmathresult:5.4cm) {};
		}
\node (1) at (v1){321};
\node (2) at (v2){142};
\node (3) at (v3){521};
\node (4) at (v4){341};
\node (5) at (v5){351};
\node (6) at (v6){541};
\node (7) at (v7){342};
\node (8) at (v8){523};
\node (9) at (v9){542};
\node (10) at (v10){543};

\draw[EdgeStyle][red] (1) to node[LabelStyle, near end]{$1$} (3);
\draw[EdgeStyle][red] (1) to node[LabelStyle, near start]{$2$} (5);
\draw[EdgeStyle][red] (3) to node[LabelStyle, near start]{$3$} (8);
\draw[EdgeStyle][blue] (2) to node[LabelStyle]{$1$} (9);
\draw[EdgeStyle][blue] (3) to node[LabelStyle, pos=0.36]{$2$} (6);
\draw[EdgeStyle][blue] (6) to node[LabelStyle]{$3$} (9);
\draw[EdgeStyle][green] (2) to node[LabelStyle, pos=0.4]{$1$} (7);
\draw[EdgeStyle][green] (1) to node[LabelStyle]{$2$} (4);
\draw[EdgeStyle][green] (4) to node[LabelStyle, pos=0.55]{$3$} (7);
\draw[EdgeStyle][orange] (4) to node[LabelStyle, pos=0.70]{$1$} (6);
\draw[EdgeStyle][orange] (4) to node[LabelStyle]{$2$} (5);
\draw[EdgeStyle][orange] (6) to node[LabelStyle, pos=0.68]{$3$} (10);
\draw[EdgeStyle][magenta] (7) to node[LabelStyle, near start]{$1$} (9);
\draw[EdgeStyle][magenta] (8) to node[LabelStyle]{$2$} (10);
\draw[EdgeStyle][magenta] (9) to node[LabelStyle]{$3$} (10);
\end{tikzpicture}
\caption{The flip graph of a $(5,3)$-matching field. Each coloured subgraph is the embedding of the linkage tree of the $4$-subset. Each edge is labelled by the deviating position.}%
\label{fig:flip+graph}%
\end{figure}
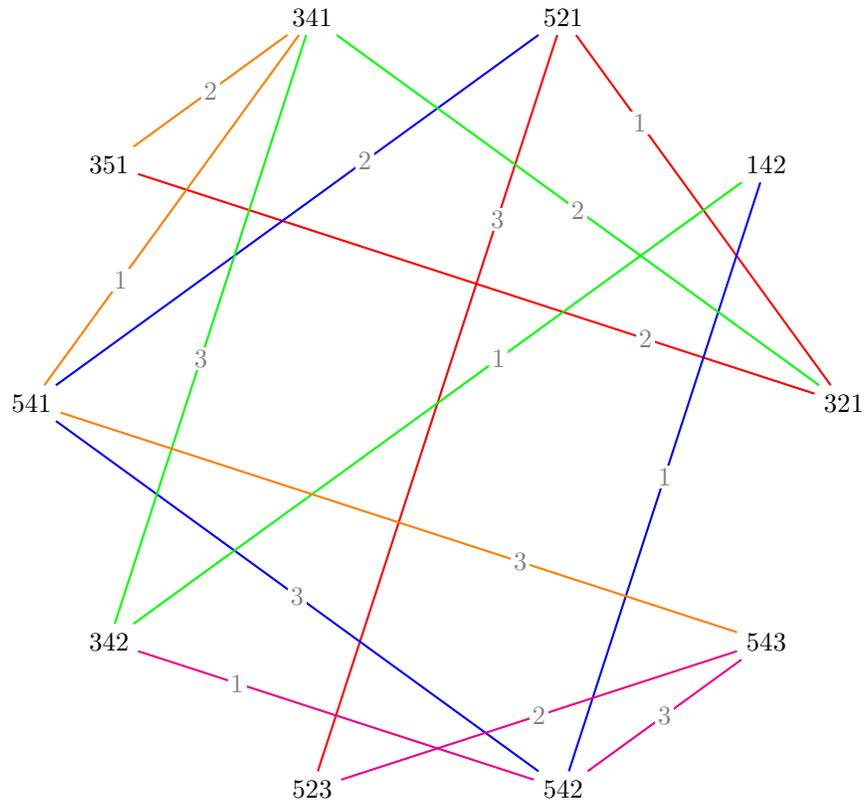

The convex hull of the characteristic vectors of these matchings in $\RR^{5\times 3}$ is the matching field polytope of the $(5,3)$-matching field.
It has $35$ edges and $f$-vector $(10,35,61,59,32,9)$.
The adjacencies of its vertex edge graph are
\[
\begin{array}{@{}*{15}{c}@{}}
341: & 351 & 541 & 342 & 543 & 321 & 142 \\
351: & 341 & 541 & 342 & 523 & 542 & 543 & 321 & 142 & 521 \\
541: & 341 & 351 & 542 & 543 & 142 & 521 \\
342: & 341 & 351 & 523 & 542 & 543 & 321 & 142 \\
523: & 351 & 342 & 542 & 543 & 321 & 142 & 521 \\
542: & 351 & 541 & 342 & 523 & 543 & 142 & 521 \\
543: & 341 & 351 & 541 & 342 & 523 & 542 & 142 \\
321: & 341 & 351 & 342 & 523 & 142 & 521 \\
142: & 341 & 351 & 541 & 342 & 523 & 542 & 543 & 321 & 521 \\
521: & 351 & 541 & 523 & 542 & 321 & 142 
\end{array} \enspace .
\]
This was computed with \texttt{polymake} \cite{DMV:polymake}.
In particular, the graph in Figure~\ref{fig:flip+graph} is a proper subgraph of the vertex-edge graph.
\end{example}

For a linkage $(n,d)$-matching field $\mathcal{M}$, the flip graph has some nice properties.

\begin{lemma}
The flip graph of an $(n,d)$-linkage matching field has $\binom{n}{d+1} \cdot d$ edges.
\end{lemma}
\begin{proof}
  The edges correspond to those topes whose right degree vector is a permutation of the partition $(2^1,1^{d-1})$.
  By Theorem~\ref{thm:constructed+topes}, there are exactly $d$ such topes for each linkage covector, each of which are distinct.
  Since there are $\binom{n}{d+1}$ linkage covectors, the claim follows.
\end{proof}

More generally we obtain a characterisation of the topes in terms of subgraphs of the flip graph. 

\begin{proposition} \label{prop:topes+products+simplices}
  A tope with right degree vector $v$ is the union of the $v_1 \cdots v_d$ matchings on the sets $N_1 \times \dots \times N_d$, where $N_i$ are the nodes adjacent with $r_i$ in the tope.
  
  Conversely, let $U$ be a subset of the matchings such that the induced subgraph of the flip graph on $U$ is the vertex-edge graph of a product of simplices $\Delta_{v_1-1} \times \dots \times \Delta_{v_d-1}$, where $v_i \geq 1$ for all $i \in [d]$.
  Then the union of the matchings in $U$ is a tope with right degree vector $(v_1,\ldots,v_d)$.
\end{proposition}

\begin{lemma} \label{lem:quadrangle+subgraph}
Let $\mu_{1}^{(1)},\mu_{1}^{(2)},\mu_{2}^{(1)},\mu_{2}^{(2)}$ be matchings such that the induced subgraph on their corresponding vertices in the flip graph is a quadrangle. Then there exist two distinct nodes $\rno_p,\rno_q \in R$ such that $\mu_{m}^{(1)}, \mu_{m}^{(2)}$ agree outside of $\rno_p$ and $\mu_{1}^{(m)}, \mu_{2}^{(m)}$ agree outside of $\rno_q$ for $m = 1,2$.
\end{lemma}
\begin{figure}%
\begin{tikzpicture}[scale=0.6]
\node[VertexStyle] (11) at (0,0){$\mu_1^{(1)}$};
\node[VertexStyle] (12) at (6,0){$\mu_1^{(2)}$};
\node[VertexStyle] (21) at (0,4){$\mu_2^{(1)}$};
\node[VertexStyle] (22) at (6,4){$\mu_2^{(2)}$};
\draw (11) -- (3,0) node[below] {$p$} -- (12);
\draw (11) -- (0,2) node[left] {$q$} -- (21);
\draw (21) -- (3,4) node[above] {$s$} -- (22);
\draw (12) -- (6,2) node[right] {$t$} -- (22);
\end{tikzpicture}
\caption{Quadrangle of matchings from Lemma~\ref{lem:quadrangle+subgraph}.}%
\label{fig:quad+matchings}%
\end{figure}
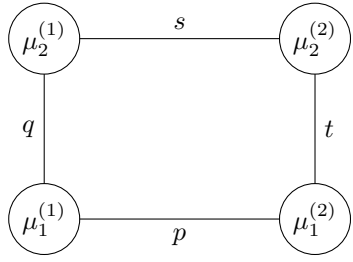
\begin{proof}
	By definition of the flip graph, each pair of adjacent matchings agree outside of a single node.
	Let Figure \ref{fig:quad+matchings} be the induced subgraph of the matchings where the edge labels denote which node they differ in. 
	We deduce that $\mu_{1}^{(1)}, \mu_{2}^{(2)}$ must agree outside of two nodes, specifically $\{\rno_p,\rno_q\}$ and $\{\rno_s,\rno_t\}$, therefore these two sets must be equal.
	Observe that any two adjacent edges in the flip graph must correspond to different right nodes, else they form a $3$-cycle between their vertices, contradicting the quadrangle as our induced subgraph.
	Therefore $\rno_p = \rno_s$ and $\rno_q = \rno_t$.
	
\end{proof}

\begin{proof}[Proof of Proposition \ref{prop:topes+products+simplices}]
  Given a tope on $\sigma \subseteq \lnoset$, the restriction of the linkage matching field to $\sigma$ is also linkage.
  The first part follows directly from Theorem~\ref{thm:constructed+topes} by applying it to the linkage matching field restricted to $\sigma$.
  
  We prove the second part by induction.
  If the induced subgraph on $U$ is the vertex-edge graph of a product of simplices with only one non-trivial factor, it is the vertex-edge graph of a simplex and so all the matchings in $U$ can only differ in the edges incident with the same node.
	Hence, their union is a tope.
  
  Assume that the induced subgraph is the vertex-edge graph of a product of simplices where $k \geq 2$ factors are non-trivial.
  Without loss of generality let these be the first $k$ factors.
  Then $U$ decomposes into $v_k$ disjoint sets $U_1, \ldots, U_{v_k}$ corresponding to faces of the product such that the induced subgraph on $U_{i}$ is isomorphic to the graph of $\prod_{j \in [k-1]}\Delta_{v_j-1}$ for all $i \in [v_k]$.
  By induction the union of the matchings in $U_i$ forms a tope $T_i$ whose right degree vector is a permutation of $(v_1,\ldots,v_{k-1},1,\ldots,1)$.
  In particular, there is a $(k-1)$-set $\sigma_i$, such that every matching contained in $T_i$ agrees on $R \setminus \sigma_i$.
  Note that $T_i$ and $U_i$ contain exactly the same matchings as subgraphs and as elements respectively.

  We claim that for all $i,j \in [v_k]$ the topes $T_i,T_j$ differ in a single node of degree one.
	As the induced subgraph of $U_i \cup U_j$ is isomorphic to $(\prod_{j \in [k-1]}\Delta_{v_j-1}) \times \Delta_1$, for any matchings $\mu_i^{(1)},\mu_i^{(2)} \in U_i$ that differ by a flip, there exists $\mu_j^{(1)},\mu_j^{(2)} \in U_j$ such that the induced subgraph on their corresponding vertices in the flip graph is a quadrangle.
	By Lemma \ref{lem:quadrangle+subgraph}, we draw two conclusions: that $\mu_i^{(1)},\mu_j^{(1)}$ and $\mu_i^{(2)},\mu_j^{(2)}$ differ in the same node and that $\mu_i^{(1)},\mu_i^{(2)}$ and $\mu_j^{(1)},\mu_j^{(2)}$ do also.
	Iterating this over all pairs of matchings that differ by a flip, the first statement implies that $T_i,T_j$ differ in one node, while the second implies it must be a node of degree one.
	If this was not the case, a node of degree two would form a $3$-cycle with any pair of matchings in $U_i$ that agree outside of that node, breaking the quadrangle.

	We obtain that the topes $T_i$ all differ in a flip of an edge incident with the same node.
As each $T_i$ has a different neighbour, the union of the $T_i$ gives a tope with right degree vector $(v_1,\ldots,v_d)$.
\end{proof}

The occurrence of all the vertex-edge graphs of products of simplices in the flip graph can be used to define an interesting cell complex.

\begin{definition}[{\cite[Section 9.2.1]{Kozlov:2008}}]
Let $G$ be an arbitrary graph.
The \emph{prodsimplicial flag complex} $PF(G)$ of $G$ is defined as follows: the graph $G$ is taken to be the $1$-skeleton of $PF(G)$, and the higher-dimensional cells are taken to be all those products of simplices whose $1$-dimensional skeleton is contained in the graph $G$.
\end{definition}

The prodsimplicial complex is an object from combinatorial algebraic topology that allow for more flexibility that simplicial complexes but more structure than an arbitrary cell complex.
For example, they generalise both simplicial and cubical complexes.
Furthermore, the prodsimplicial flag complex can be viewed as a direct generalisation of the clique complex of $G$.

Proposition~\ref{prop:topes+products+simplices} and Theorem~\ref{thm:constructed+topes} directly imply the next statement.

\begin{theorem} \label{thm:prodsimplicial+complex+flip+graph}
  The prodsimplicial flag complex of the flip graph of a linkage $(n,d)$-matching field has the same face lattice as the set of topes derived from it.
  In particular, the maximal cells of the prodsimplicial flag complex are in bijection with $\simplat{n-d}{d}$.
\end{theorem}

\section{Triangulations of $\Dprod{n-1}{d-1}$ and pairs of lattice points}  \label{sec:pairs+of+lattice+points}

We denote the set of triangulations of $\Dprod{n-1}{d-1}$ by $\cT\cS(n,d)$. By \cite[Theorem 6.2.13]{DeLoeraRambauSantos} there are exactly $K = \binom{n+d-2}{n-1}$ full-dimensional simplices in such a triangulation.

Given a set of trees encoding a triangulation of $\Delta_{n-1}\times\Delta_{d-1}$, consider the map that sends a tree $T$ to its left and right degree vector pair $(u,v)$.
By Lemma~\ref{lem:degree-vectors-trees}, this map is injective.
As each left and right degree vector can be identified with a lattice point in $\simplat{d-1}{n}$ and $\simplat{n-1}{d}$ by subtracting $\unit{[n]}$ and $\unit{[d]}$ respectively, this map can be written as

\[
\phi_{n,d} \colon T \mapsto (u - \unit{[n]},v - \unit{[d]}) \in \simplat{d-1}{n} \times \simplat{n-1}{d} \enspace .
\]
This induces the map
\begin{equation} \label{eq:degree-pair-map}
\Phi_{n,d} \colon \cT\cS(n,d) \rightarrow \binom{\simplat{d-1}{n} \times \simplat{n-1}{d}}{K} \enspace ,
\end{equation}
where each tree describing a full-dimensional simplex in the triangulation is mapped to its left and right degree vector pair minus $\unit{[n]}$ and $\unit{[d]}$ respectively.

After \cite[Theorem 12.9]{Postnikov:2009}, Postnikov asked whether the map defined in \eqref{eq:degree-pair-map} is injective.
His question is posed for root polytopes in general.
As discussed earlier, parallel to our work, Galashin, Nenashev and Postnikov derived an affirmative answer to this question in~\cite{GalashinNenashevPostnikov2018}.
Whereas their approach is based on the newly introduced notion of trianguloids, we present an independent proof for the case of the full product of two simplices (but not root polytopes in general) by exploiting the same structure as for Theorem~\ref{thm:lattice-points-Chow}.

A natural approach is to reconstruct the triangulation using the structure of the dual graph.
However, this does not determine the triangulation as we illustrate.

\begin{example}[Triangulation not determined by dual graph]
  Figure~\ref{fig:dual+graph} shows two triangulations of $\Delta_{2}\times\Delta_{2}$ whose dual graphs are the same.
	Furthermore, the trees encoding their triangulations contain all of the same non-maximal matchings.
	However, their unique maximal matchings are not equal, and so the triangulations are not equal, even up to symmetry.
  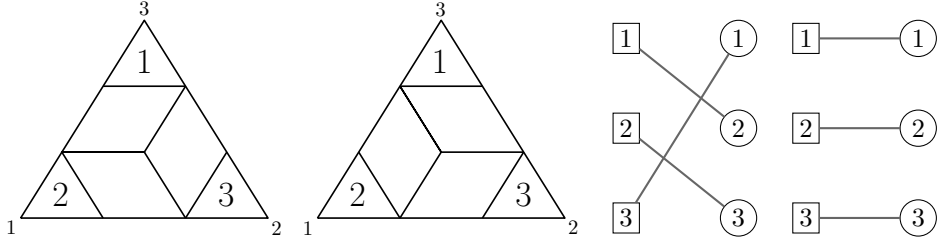
\begin{figure}%
\centering
\includegraphics[width=0.6\textwidth]{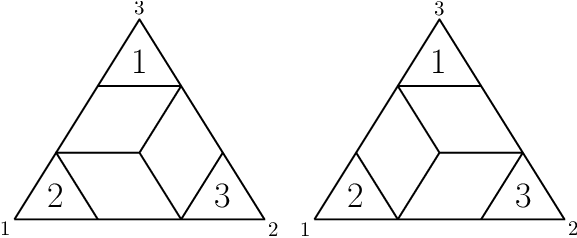}
\quad
\resizebox{0.35\textwidth}{!}{\begin{tikzpicture}[scale=0.6]
	\node[BoxVertex] (1) at (0,2){1};
	\node[BoxVertex] (2) at (0,0){2};
	\node[BoxVertex] (3) at (0,-2){3};
	\node[VertexStyle] (11) at (2.5,2){1};
	\node[VertexStyle] (22) at (2.5,0){2};
	\node[VertexStyle] (33) at (2.5,-2){3};
	
	\draw[EdgeStyle] (1) -- (22);
	\draw[EdgeStyle] (2) -- (33);
	\draw[EdgeStyle] (3) -- (11);
	
	\node[BoxVertex] (1) at (4,2){1};
	\node[BoxVertex] (2) at (4,0){2};
	\node[BoxVertex] (3) at (4,-2){3};
	\node[VertexStyle] (11) at (6.5,2){1};
	\node[VertexStyle] (22) at (6.5,0){2};
	\node[VertexStyle] (33) at (6.5,-2){3};
	
	\draw[EdgeStyle] (1) -- (11);
	\draw[EdgeStyle] (2) -- (22);
	\draw[EdgeStyle] (3) -- (33);
\end{tikzpicture}}
\caption{The mixed subdivisions corresponding to two triangulations of $\Delta_2 \times \Delta_2$. Both have the same dual graph, and the same non-maximal matchings. However their unique maximal matchings are different.}
\label{fig:dual+graph}%
\end{figure}
\end{example}

\begin{figure}%
\includegraphics[width=0.7\textwidth]{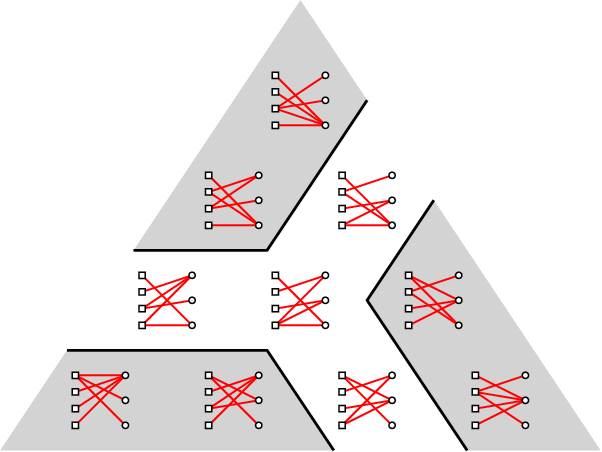}%
\caption{The trees encoding a triangulation of $\Dprod{3}{2}$ arranged by their bijection to $\simplat{4}{2}$.
The grey regions are the (open) combinatorial sectors $S_4^{(1)}, \, S_4^{(2)}, \, S_4^{(3)}$.}%
\label{img+triangulation+sector}%
\end{figure}

Definition~\ref{def:combinatorial+sector} introduces combinatorial sectors for a collection of compatible bipartite graphs $\cG$ with a bijective map to lattice points of a simplex.
The trees encoding a triangulation are all compatible and have a natural bijection to both $\simplat{d-1}{n}$ and $\simplat{n-1}{d}$ determined by their left and right degree vectors.
When $\cG$ is a triangulation $\cT$, combinatorial sectors are a direct analogue to (open) sectors of tropical hyperplane arrangements as described in Example \ref{ex:hyperplane+sector}, see Figure \ref{img+triangulation+sector} for an example.
	We use this notion alongside an iterative method, analogous to the proof of Theorem~\ref{thm:lattice-points-Chow}, to construct the triangulation inductively from the triangulations of its faces.
	
\begin{theorem} \label{thm:phi+injective}
  The map $\Phi_{n,d}$ is injective.
\end{theorem}
\begin{proof}
  We proceed by induction on $n+d$. For $n = d = 1$ there is only one tree.
  Now consider a triangulation $\mathcal{T}$ represented by a collection of trees for $n+d > 2$.
	Assume that any triangulation of $\Dprod{j-1}{i-1}$ is uniquely determined by the degree vector pairs of its trees for $j+i < n+d$.
  We show the case $n \geq d$, the case $n \leq d$ is entirely analogous.
  Each left degree vector contains an entry equal to $1$ since
  \[
  2\cdot (n-1) + 1 = 2n-2 +1 = 2n -1 \geq n +d-1 \enspace .
  \]
Hence, we get a non-disjoint decomposition
\[
\bigcup_{j \in [n]} \mathcal{L}_j = \Phi_{n,d}(\mathcal{T}) \quad \mbox{for} \quad \mathcal{L}_j = \{(u,v) \mid u_j = 1\} \enspace .
\]
Now fix a $j \in [n]$. There is a partition of $\mathcal{L}_j$ in the sets
\[
\mathcal{L}_j^{(i)} = \{ (u,v) \mid (\lno_j,\rno_i) \in G,\;G \in \mathcal{T}~\mbox{has degree vector}~(u,v)~\mbox{with}~u_j=1\}
\]
where $\mathcal{L}_j^{(i)}$ is the image of $\cS_j^{(i)}$ in $\Phi_{n,d}$.

From $\mathcal{L}_j^{(i)}$ we can construct a set $\overline{\mathcal{L}_j^{(i)}}$ by removing the $j$th entry of the first component and decreasing the $i$th entry of the second component for all the pairs.
This corresponds to removing the leaf edge $(\lno_j,\rno_i)$ of the trees.
The resulting set
\[
\overline{\mathcal{L}_j} = \bigcup_{i \in [d]} \overline{\mathcal{L}_j^{(i)}}
\]
is the set of degree vectors of the deletion of $\mathcal{T}$ with respect to $j$. Hence, we can apply induction and deduce that $\overline{\mathcal{L}_j}$ uniquely defines the trees with the contained degree vectors. From the partition into the $\overline{\mathcal{L}_j^{(i)}}$ we can recover to which node $\lno_j$ is incident in the original tree. Therefore, we can construct all trees for which $\lno_j$ has degree $1$. Ranging over all $\lno_j \in L$, we get all trees of $\mathcal{T}$.

\smallskip

It remains to show how to construct the set $\mathcal{L}_j^{(i)}$ for each $\rno_i \in R$, which we now demonstrate.
Assume without loss of generality that $i = 1$ and apply Algorithm~\ref{algo:derive+sectors}.
\begin{algorithm}[htbp]
  \caption{Construct the degree pairs of a combinatorial sector of trees}
  \label{algo:derive+sectors}
  \begin{algorithmic}[1]
  	\Statex \textbf{Input:} $\Phi_{n,d}(\cT)$, the set of degree vector pairs of the trees $\cT$
  	\Statex \textbf{Output:} $\mathcal{L}_j^{(1)}$, the set of degree vector pairs with $u_j = 1$ whose tree contains $(\lno_j,\rno_1)$	
    \If{$u_j = 1$ for $(u,v)$ with $v_1 = n$}
    \State $\mathcal{K}_j \gets \{(u,v)\}$ \label{line:initial+K}
    \Else
    \State \Return $\emptyset$
    \EndIf
    \State $h \gets n-1$
    \While{$h > 0$}
    \ForAll{$(u,v) \in \mathcal{L}_j~\mbox{with}~v_1=h$}
    \If{$\exists k \in [d] \colon v_k > 1 \colon \exists\,w^{(k)} \colon (w^{(k)},v+e_1-e_k) \in \mathcal{K}_j$} \label{line:ancestor}
    \State $\mathcal{K}_j \gets \mathcal{K}_j \cup (u,v)$
    \EndIf
    \State $h \gets h-1$
    \EndFor
    \EndWhile 
    \State \Return $\mathcal{K}_j$
\end{algorithmic}
\end{algorithm}

\textbf{Claim:} $\mathcal{K}_j = \mathcal{L}_j^{(1)}$.

\textbf{Proof by induction} 
There is a unique tree $T_0$ with the right degree vector $\unit{[d]} + (n-1)\unit{1}$. If $u_j = 1$ then $\lno_j$ is a leaf and it is adjacent to $\rno_1$ because of the structure of the right degree vector.
Line~\ref{line:initial+K} in the algorithm guarantees that $T_0$ is in $\mathcal{K}_j$. Furthermore, the edge $(\lno_j,\rno_1)$ shows that it is also contained in $\mathcal{L}_j^{(1)}$.

Now, assume that $\mathcal{K}_j$ and $\mathcal{L}_j^{(1)}$ agree in all elements whose first entry of the second component is $h+1 \leq n$. 

Let $(u,v) \in \mathcal{L}_j$ such that $v_1 = h$ and $(w,v+e_1-e_k) \in \mathcal{K}_j$ an element fulfilling the condition in Line~\ref{line:ancestor}. These two vectors are the right degree vectors of two compatible trees $T_1$ and $T_2$. Note that $\mathcal{K}_j \subseteq \mathcal{L}_j$. As, by the induction hypothesis, $(\lno_j,\rno_1)$ is an edge of $T_2$ we can deduce with Corollary~\ref{cor:neighbour-trees} that this is also an edge of $T_1$. Hence, $(u,v)$ is an element of $\mathcal{L}_j^{(1)}$.

Conversely, let $(u,v) \in \mathcal{L}_j^{(1)}$ be with $v_1 = h$. By Lemma~\ref{lem:abstract-containment-sector}, there is a $k \in [d]$ and a $w \in \simplat{d-1}{n}$ such that in the tree with degree pair $(w + \unit{[n]},v+e_1-e_k)$ the node $\lno_j$ is a leaf and it is adjacent to $\rno_1$. The induction hypothesis implies that $(w + \unit{[n]},v+e_1-e_k) \in \mathcal{K}_j$. Now, Line~\ref{line:ancestor} shows that also $(u,v)$ is an element of $\mathcal{K}_j$.
\end{proof}

	This map is far from surjective.
	By \cite[Theorem 12.9]{Postnikov:2009}, every possible lattice point in $\simplat{d-1}{n}$ and $\simplat{n-1}{d}$ must appear precisely once, the only remaining choice is how to pair them up.
	This gives the following immediate corollary.

\begin{corollary} \label{cor:triangulation+bound}
  The number of all triangulations of $\Dprod{n-1}{d-1}$ is bounded from above by $\binom{n+d-2}{d-1}!$.
\end{corollary}

\begin{remark}
  Note that this bound is tight for $n = 2$ as by \cite[Proposition 6.2.3]{DeLoeraRambauSantos} triangulations of $\Dprod{1}{d-1}$ are in bijection with permutations of $[d]$.
  Theorem 5.4 and Corollary 5.5 in \cite{Santos:2005} give upper bounds for regular subdivisions but, to the knowledge of the authors, this is the first upper bound on the number of all triangulations. 
	Recall from \cite[Theorem 6.2.19]{DeLoeraRambauSantos} that non-regular triangulations of $\Dprod{n-1}{d-1}$ exist if and only if $(n-2)(d-2)\geq 4$.  
  Even if there are many non-regular triangulations, the bound might be very coarse as we do not use the structure of the lattice points.
\end{remark}

Theorem \ref{thm:phi+injective} leads to the following natural question.
\begin{question}
Can one formulate an axiom system for lattice point pairs arising from triangulations of $\Dprod{n-1}{d-1}$?
\end{question}
This would give a cryptomorphic axiom system to that in Proposition \ref{prop:char-triang}.
We state some necessary conditions that lattice point pairs must satisfy to induce a triangulation of $\Dprod{n-1}{d-1}$.

We denote a lattice point pair of type $(n,d)$ by $p = (u,v) \in \simplat{d-1}{n} \times \simplat{n-1}{d}$.
We say two lattice point pairs $p, p'$ are \emph{adjacent} if $u, u'$ and $v,v'$ differ by one in precisely two coordinates and differ nowhere else.
Note that every pair of trees that differ by a flip induce adjacent lattice point pairs, but the converse is not true.
By the second condition of Proposition \ref{prop:char-triang}, the number of neighbours a tree has in the flip graph is equal to the number of edges in the tree that are not leaves.

We defined the sets $\mathcal{L}_j = \SetOf{(u,v)}{u_j=1}$ in the proof of Theorem \ref{thm:phi+injective} for all $j \in [n]$.
We construct a \emph{deletion} of this set $\overline{\mathcal{L}_j}$ by removing the $j$th entry of the first component and decreasing an entry in the second component of all lattice point pairs.
Note the distinction between left and right is arbitrary and we can analogously define the set $\mathcal{L}_i = \SetOf{(u,v)}{v_i=1}$ and its deletion $\overline{\mathcal{L}_i}$.
In the proof we construct a specific deletion, but we just need to ensure a well behaved one exists.
This, along with \cite[Theorem 12.9]{Postnikov:2009}, gives the following necessary conditions on pairs of lattice points:

\begin{corollary} \label{cor:comp+necessary+cond}
Let $\cP$ be a set of lattice point pairs that induces a triangulation of $\Dprod{n-1}{d-1}$. Then $\cP$ satisfies the following conditions:
\begin{itemize}
\item The projections from $\cP$ onto $\simplat{d-1}{n}$ and $\simplat{n-1}{d}$ are bijections.
\item Each $p \in \cP$ has at least $n + d - 1 - l(p)$ adjacent lattice point pairs in $\cP$, where $l(p)$ is the number of coordinates in $p$ whose entry is $1$.
\item For every $\mathcal{L}_k \subset \cP$, there exists a deletion $\overline{\mathcal{L}_k}$ satisfying the first two conditions.
\end{itemize}
\end{corollary}

Recall that triangulations of $\Dprod{n-1}{1}$ are in bijection with the permutations in $S_n$.
By Theorem \ref{thm:phi+injective}, such a triangulation is also determined by a certain collection of lattice point pairs.
Let $\mu$ be the permutation corresponding to the triangulation.
Then the lattice point pairs are given by
\[
([k,n+1-k],[1^{\mu(k)-1},2,1^{n-\mu(k)}]) \qquad \text{ for all } k \in [n] \enspace .
\]
The lattice point on the left corresponds to an element $k \in [n]$, the lattice point on the right corresponding to the element $\mu(k)$ that it is mapped to.
This establishes the bijection with $S_n$.

Theorem \ref{thm:phi+injective} allows us to generalise this construction.
We fix orderings on $\simplat{d-1}{n}$ and $\simplat{n-1}{d}$ to establish a correspondence with $[K]$, where $K = \binom{n+d-2}{n-1}$.
Triangulations of $\Dprod{n-1}{d-1}$ now correspond to elements of the symmetric group $S_K$, where each lattice point pair determines an element of $[K]$ and where it is mapped to.
Note that we have natural $S_n$ and $S_d$ actions given by changing the ordering on the lattice points, therefore each triangulation gives rise to a subset of a conjugacy class of $S_K$.
It would be interesting to study this link to the symmetric group in more detail.

\section{Matching Stacks and Transversal Matroids} \label{sec:matching+stacks+transversal+matroids}

In some sense, matching fields contain complementary information to transversal matroids.
While transversal matroids encode on which subsets of the nodes a graph contains matchings, a matching field contains a matching for all $d$-subsets of $\lnoset$ and one is interested in the interplay of the matchings.

\begin{definition}
  A \emph{matching stack} on $\lnoset \sqcup \rnoset$ is a map which assigns to each pair $(J,I)$ with $J \subseteq \lnoset, I \subseteq \rnoset$ and $|J| = |I|$ a perfect matching $\mu$ on $J \sqcup I$ (support axiom). 
 
A matching stack is a \emph{matching ensemble} if the following axioms are satisfied:
  \begin{itemize}
\item each submatching $\eta$ of $\mu$ on subsets $J' \subset J$, $I' \subset I$ is the image of $(J',I')$ (closure axiom).  
\item the matchings for all fixed $J \subseteq \lnoset$ and for all fixed $I \subseteq \rnoset$ form linkage matching fields (linkage axiom).
\end{itemize}
\end{definition} 

Matching ensembles were first studied in~\cite{OhYoo-ME:2013}.
The main result in~\cite{OhYoo-ME:2013} is the equivalence of triangulations of $\Delta_{n-1} \times \Delta_{d-1}$ and matching ensembles on $\lnoset \sqcup \rnoset$. We present an intermediate result to demonstrate further research directions. 

The tropical Stiefel map, extensively studied in \cite{FinkRincon:2015}, produces a \emph{tropical linear space} for a matrix over the tropical semiring $\TT = (\RR\cup\{-\infty\},\max,+)$.
Explicitly, the map takes a tropical matrix $A$ in $\TT^{d\times n}$ and associates its \emph{tropical Pl\"{u}cker vector} $p$ in $\TT^{\binom{n}{d}}$.
The coordinates of $p$ are given by 
\[
p_I\ =\ \max_{\sigma \in S_d}\left(\sum_{\ell = 1}^{d}A_{i_{\ell}, \sigma(\ell)}\right)
\]
for each subset $I = \{i_1,\dots,i_d\} \in \binom{[n]}{d}$, where the maximum is taken over the symmetric group $S_d$.
This gives rise to a matroid subdivision of the hypersimplex $\Delta_{d,n}$ which is dual to a \emph{Stiefel tropical linear space}.
Each cell in the regular matroid subdivision arising from such a height function is the matroid polytope of a transversal matroid.
In particular, these matroids are the transversal matroids of the covector graphs for the cells in the regular subdivision of $\Dprod{n-1}{d-1}$ induced by $A$. 

We give a combinatorial sketch of a similar polyhedral construction from \cite[Theorem 7]{HerrmannJoswigSpeyer:2014} . 
They start with the bipartite graphs corresponding to the full-dimensional cells in a not-necessarily regular subdivision of $\Delta_{k-1} \times \Delta_{d-1}$ where $k = n-d$.
They augment the left node set of the bipartite graphs by $d$ dummy nodes.
By connecting all nodes in $\rnoset$ of degree $1$ in each graph $G$ to the corresponding dummy node, they ensure that the transversal matroid on $\lnoset$ has no loops.
These two constructions motivates the following, more general construction.

\begin{definition}
  Let $\cG$ be a collection of bipartite graphs on the same node set $\lnoset \sqcup \rnoset$.
	The \emph{combinatorial Stiefel map} associates to each graph $G$ in $\cG$ its corresponding transversal matroid $M_G$.
\end{definition}

With the construction for Theorem~\ref{thm:constructed+topes}, we can start with a linkage matching field and construct trees in a natural way for all degree vectors.
Recall that the topes are compatible with the matching field but the tope linkage covectors may not be. 
Taking the combinatorial Stiefel map of the collection of the maximal tope linkage covectors results in a collection of transversal matroids. Inspired by \cite[Corollary 5.6]{FinkRincon:2015} and \cite[Theorem 7]{HerrmannJoswigSpeyer:2014} we conjecture the following.

\begin{conjecture}
  A linkage matching field is determined by the collection of transversal matroids associated to the maximal tope linkage covectors by the combinatorial Stiefel map.
\end{conjecture}

\begin{remark}
  As maximal topes arising from a linkage matching field play an important role, one should keep in mind that the combinatorial Stiefel image of a tope is just a partition matroid.
  However, in the tropical Stiefel map which gives rise to a matroid subdivision of the hypersimplex, they do not correspond to maximal cells.
\end{remark}

The latter conjecture transfers the connection between regular subdivisions of $\Dprod{n-1}{d-1}$ and regular matroid subdivisions of the hypersimplex $\Delta_{d,n}$ to the purely combinatorial setting of matching fields.
It is motivated by the interplay between matching fields, matching stacks and triangulations of $\Dprod{n-1}{d-1}$ which we now discuss in more detail.

We introduce two ways to move between matching fields and matching stacks: completion and extension.
Their relationship is shown in Figure \ref{fig:fields+stacks}.
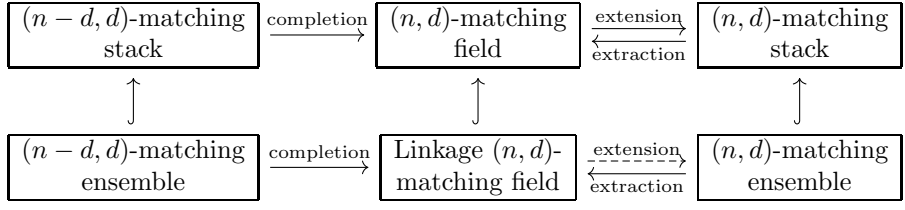
\begin{figure}
\begin{tikzcd}[column sep=large]
\begin{tabular}{|c|} \hline $(n-d,d)$-matching \\ stack \\ \hline \end{tabular} \arrow{r}{\text{completion}} & \begin{tabular}{|c|} \hline $(n,d)$-matching \\ field \\ \hline \end{tabular} \arrow[shift left]{r}{\text{extension}} & \begin{tabular}{|c|} \hline $(n,d)$-matching \\ stack \\ \hline \end{tabular} \arrow[shift left]{l}{\text{extraction}} \\
\begin{tabular}{|c|} \hline $(n-d,d)$-matching \\ ensemble \\ \hline \end{tabular} \arrow{r}{\text{completion}} \arrow[u,hook] & \begin{tabular}{|c|} \hline Linkage $(n,d)$- \\ matching field \\ \hline \end{tabular} \arrow[dashed,shift left]{r}{\text{extension}} \arrow[u,hook] & \begin{tabular}{|c|} \hline $(n,d)$-matching \\ ensemble \\ \hline \end{tabular} \arrow[u,hook] \arrow[shift left]{l}{\text{extraction}}
\end{tikzcd}
\caption{The relationships between different classes of matching fields and matching stacks.}
\label{fig:fields+stacks}
\end{figure}
Let $\mathcal{S}$ be a matching stack on $\lnoset \sqcup \rnoset$.
We introduce $d$ dummy nodes to get the left node set $\hat{\lnoset} = \lnoset \cup \SetOf{\lno_{n+i}}{i \in [d]}$.
To the matching $\mu$ on $J \sqcup I$ in $\mathcal{S}$ we associate the matching $\hat{\mu}$ defined as follows:
\[
\hat{\mu}(\lno_j) = \mu(\lno_j) \quad \mbox{ for } \lno_j \in J \quad \mbox{and} \quad \hat{\mu}(\lno_{n+h}) = \rno_h \quad \mbox{for all~} \rno_h \in \rnoset \setminus I \enspace .
\]
This yields a matching field on $\hat{\lnoset} \sqcup \rnoset$, which we call the \emph{matching field completion} of $\cS$.
This is a pointed matching field in the sense of \cite[Example 1.4]{SturmfelsZelevinsky:1993}.
Note that the resulting matching fields have a very particular structure, therefore there is no inverse construction in general.

\begin{theorem} \label{thm:linkage-dummy-extension}
  Let $\mathcal{M}$ be the matching field completion of the matching stack $\mathcal{S}$. Then $\mathcal{M}$ fulfils left linkage if and only if $\mathcal{S}$ fulfils left linkage. 
\end{theorem}
\begin{proof}
  Consider a linkage covector $C$ of $\cS$ on the node set $J \sqcup I$ such that $|J| = |I|+1$.
	We will construct the linkage covector $D$ of $\cM$ on the node set $\hat{J} \sqcup \rnoset$ where $\hat{J} = J \cup \SetOf{\lno_{n+i}}{\rno_i \notin I}$.
	The matching $M_j$ on $\hat{J} \setminus \{\lno_j\}$, where $\lno_j \in J$, is the union of the matching in $C$ obtained by isolating $\lno_j$ with the edges $\SetOf{(\lno_{n+i},\rno_i)}{\rno_i \notin I}$.
	The matching $M_{n+i}$ on $\hat{J} \setminus \{\lno_{n+i}\}$, where $\rno_i \notin I$, is the union of the matching in $\cS$ on the node set $J \sqcup (I \cup \{\rno_i\})$ with the edges $\SetOf{(\lno_{n+k},\rno_k)}{\rno_k \notin I \cup \{\rno_i\}}$.
	Note that if the edge $(\lno_j,\rno_i)$ is in $M_{n+i}$, the submatching obtained by deleting it is contained in $C$.
	Therefore $D$ is the union of $C$ and pairs of edges $\SetOf{(\lno_j,\rno_i),(\lno_{n+i},\rno_i)}{\rno_i \notin I, (\lno_j,\rno_i) \in M_{n+i}}$, so it is a tree with degree two on all right nodes.
	Any linkage covector in $\cM$ can be constructed this way and so it must satisfy the linkage property.

	Conversely, consider a linkage covector of $\mathcal{M}$ on the node set $\hat{J} \sqcup \rnoset$.
	Removing the edges incident with the dummy nodes yields a tree on $J \sqcup \rnoset$ in which nodes in $I$ have degree $2$ and nodes in $\rnoset \setminus I$ have degree $1$.
	Deleting the edges adjacent to nodes in $\rnoset \setminus I$ gives the union of matchings on $J \sqcup I$ arising in $\cS$.
	As we have only removed leaves from a tree, this resulting graph is also a tree with all right nodes degree 2.
\end{proof}
A corollary to this result is that the matching field completion of a matching ensemble is always a linkage matching field, as shown in Figure \ref{fig:fields+stacks}.
Note that it is not a necessary condition, as a matching stack need not satisfy the closure axiom to give rise to a linkage matching field via completion.

Matching field completion is the analogous construction to the construction in \cite{HerrmannJoswigSpeyer:2014} for triangulations of $\Dprod{n-d-1}{d-1}$.
Their construction was a canonical way to embed $\Dprod{n-d-1}{d-1}$ as a root polytope in $\RR^{n+d}$.
However, as a subtriangulation of a triangulation of $\Dprod{n-1}{d-1}$, in general it does not give enough information to determine the whole triangulation of $\Dprod{n-1}{d-1}$.
This can be seen from our construction: an $(n-d,d)$-matching stack or ensemble can be completed to an $(n,d)$-matching field, but this does not live in a unique $(n,d)$-matching stack or ensemble.
In the latter case, it may not live in any matching ensemble, as we describe now.

The other construction we consider is \emph{extension}.
Given any $(n,d)$-matching field, we can always extend it to an $(n,d)$-matching stack by adding in arbitrary matchings for all $J \sqcup I$ with $|J|=|I|$.
Conversely, given an $(n,d)$-matching stack one can obtain an $(n,d)$-matching field by extracting the $d\times d$ matchings of the matching stack.
Note that this is analogous to the extraction method for triangulations of $\Dprod{n-1}{d-1}$, as they give rise to matching ensembles.

The picture is not so clear for linkage matching fields and matching ensembles.
Given an $(n,d)$-matching ensemble, we've seen extraction gives rise to a polyhedral $(n,d)$-matching field, a subset of linkage matching fields defined by coming from a matching ensemble.
However, we do not know the intrinsic characteristics of polyhedral matching fields, leading to the following question.

\begin{question}
When can a linkage $(n,d)$-matching field be extended to an $(n,d)$-matching ensemble?
\end{question}

We make a step in this direction with the following conjecture.
We say that a matching field $\mathcal{M}$ satisfies the \emph{compatible right submatching property} if and only if the following holds:

Let $\mu_1, \ldots, \mu_r$ be submatchings of matchings in $\mathcal{M}$ on $J \sqcup I_1, \ldots, J \sqcup I_r$ where $I := \bigcup_{[r]} I_s \subseteq \rnoset$ with $|J|+1 = |I|$.
Then $T = \bigcup_{[r]} \mu_s$ is a forest on $J \sqcup I$ and each matching $\mu$ of size $|J|$ in $T$ is compatible with the matchings in $\mathcal{M}$.

\begin{conjecture}
A (left) linkage matching field is polyhedral if and only if it satisfies the compatible right submatching property.
\end{conjecture}

\begin{figure}[htb]
  \begin{center}
	\resizebox{\textwidth}{!}{
  \begin{tikzpicture}[scale=0.6]
\bigraphtwofourcoord{0}{0}{0.6}{1.2}{2.5};
\draw[EdgeStyle] (v1) to (w1);
\draw[EdgeStyle] (v2) to (w2);
\node[BoxVertex] (1) at (v1){5};
\node[BoxVertex] (2) at (v2){6};
\node[VertexStyle] (11) at (w1){1};
\node[VertexStyle] (22) at (w2){2};
\node[VertexStyle] (33) at (w3){3};
\node[VertexStyle] (44) at (w4){4};

\bigraphtwofourcoord{4}{0}{0.6}{1.2}{2.5};
\draw[EdgeStyle] (v1) to (w3);
\draw[EdgeStyle] (v2) to (w1);
\node[BoxVertex] (1) at (v1){5};
\node[BoxVertex] (2) at (v2){6};
\node[VertexStyle] (11) at (w1){1};
\node[VertexStyle] (22) at (w2){2};
\node[VertexStyle] (33) at (w3){3};
\node[VertexStyle] (44) at (w4){4};

\bigraphtwofourcoord{8}{0}{0.6}{1.2}{2.5};
\draw[EdgeStyle] (v1) to (w2);
\draw[EdgeStyle] (v2) to (w3);
\node[BoxVertex] (1) at (v1){5};
\node[BoxVertex] (2) at (v2){6};
\node[VertexStyle] (11) at (w1){1};
\node[VertexStyle] (22) at (w2){2};
\node[VertexStyle] (33) at (w3){3};
\node[VertexStyle] (44) at (w4){4};

\bigraphtwofourcoord{12}{0}{0.6}{1.2}{2.5};
\draw[EdgeStyle] (v1) to (w1);
\draw[EdgeStyle] (v2) to (w4);
\node[BoxVertex] (1) at (v1){5};
\node[BoxVertex] (2) at (v2){6};
\node[VertexStyle] (11) at (w1){1};
\node[VertexStyle] (22) at (w2){2};
\node[VertexStyle] (33) at (w3){3};
\node[VertexStyle] (44) at (w4){4};

\bigraphtwofourcoord{16}{0}{0.6}{1.2}{2.5};
\draw[EdgeStyle] (v1) to (w4);
\draw[EdgeStyle] (v2) to (w2);
\node[BoxVertex] (1) at (v1){5};
\node[BoxVertex] (2) at (v2){6};
\node[VertexStyle] (11) at (w1){1};
\node[VertexStyle] (22) at (w2){2};
\node[VertexStyle] (33) at (w3){3};
\node[VertexStyle] (44) at (w4){4};

\bigraphtwofourcoord{20}{0}{0.6}{1.2}{2.5};
\draw[EdgeStyle] (v1) to (w3);
\draw[EdgeStyle] (v2) to (w4);
\node[BoxVertex] (1) at (v1){5};
\node[BoxVertex] (2) at (v2){6};
\node[VertexStyle] (11) at (w1){1};
\node[VertexStyle] (22) at (w2){2};
\node[VertexStyle] (33) at (w3){3};
\node[VertexStyle] (44) at (w4){4};

\bigraphtwofourcoord{2}{-5}{0.6}{1.2}{2.5};
\draw[EdgeStyle] (v1) to (w1);
\draw[EdgeStyle] (v1) to (w2);
\draw[EdgeStyle] (v1) to (w3);
\draw[EdgeStyle] (v2) to (w1);
\draw[EdgeStyle] (v2) to (w2);
\draw[EdgeStyle] (v2) to (w3);
\node[BoxVertex] (1) at (v1){5};
\node[BoxVertex] (2) at (v2){6};
\node[VertexStyle] (11) at (w1){1};
\node[VertexStyle] (22) at (w2){2};
\node[VertexStyle] (33) at (w3){3};
\node[VertexStyle] (44) at (w4){4};

\bigraphtwofourcoord{7.3}{-5}{0.6}{1.2}{2.5};
\draw[EdgeStyle] (v1) to (w1);
\draw[EdgeStyle] (v1) to (w4);
\draw[EdgeStyle] (v2) to (w2);
\draw[EdgeStyle] (v2) to (w4);
\node[BoxVertex] (1) at (v1){5};
\node[BoxVertex] (2) at (v2){6};
\node[VertexStyle] (11) at (w1){1};
\node[VertexStyle] (22) at (w2){2};
\node[VertexStyle] (33) at (w3){3};
\node[VertexStyle] (44) at (w4){4};

\bigraphtwofourcoord{12.6}{-5}{0.6}{1.2}{2.5};
\draw[EdgeStyle] (v1) to (w1);
\draw[EdgeStyle] (v1) to (w3);
\draw[EdgeStyle] (v2) to (w1);
\draw[EdgeStyle] (v2) to (w4);
\node[BoxVertex] (1) at (v1){5};
\node[BoxVertex] (2) at (v2){6};
\node[VertexStyle] (11) at (w1){1};
\node[VertexStyle] (22) at (w2){2};
\node[VertexStyle] (33) at (w3){3};
\node[VertexStyle] (44) at (w4){4};

\bigraphtwofourcoord{18}{-5}{0.6}{1.2}{2.5};
\draw[EdgeStyle] (v1) to (w2);
\draw[EdgeStyle] (v1) to (w3);
\draw[EdgeStyle] (v1) to (w4);
\draw[EdgeStyle] (v2) to (w2);
\draw[EdgeStyle] (v2) to (w3);
\draw[EdgeStyle] (v2) to (w4);
\node[BoxVertex] (1) at (v1){5};
\node[BoxVertex] (2) at (v2){6};
\node[VertexStyle] (11) at (w1){1};
\node[VertexStyle] (22) at (w2){2};
\node[VertexStyle] (33) at (w3){3};
\node[VertexStyle] (44) at (w4){4};
\end{tikzpicture}
	}
  \caption{The submatching field on $\{\lno_5,\lno_6\}\sqcup\{\rno_1,\dots,\rno_4\}$ with all possible linkage covectors of matchings. Two of the linkage covectors contain cycles and so are not right linkage.}
  \label{fig:two+four+linkage+cycle}
	\end{center}
\end{figure}
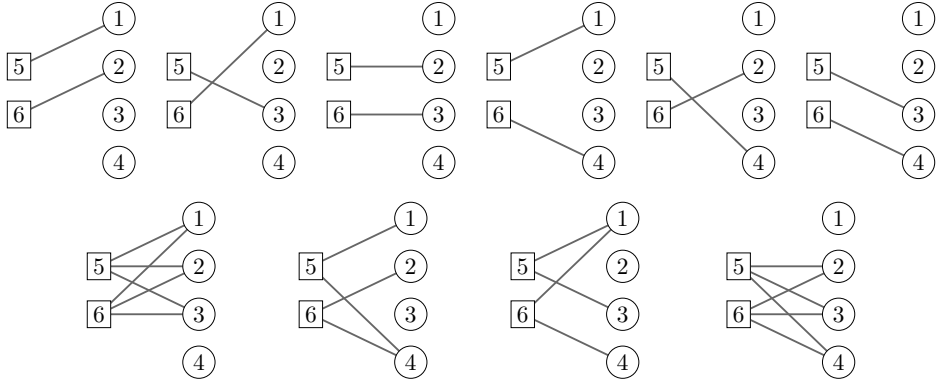

\begin{example} Continuation from Example~\ref{ex:incompatible+linkage+covectors}.
  The $2 \times 2$ matchings on $\{\lno_5,\lno_6\}$ given in Figure~\ref{fig:two+four+linkage+cycle} are submatchings of matchings in $\cM$. Therefore, any matching stack that extends $\cM$ must contain them.
  However, the $2 \times 3$ linkage covectors on $\{\rno_1,\rno_2,\rno_3\}$ and $\{\rno_2,\rno_3,\rno_4\}$ contain cycles and so any matching stack extending $\cM$ cannot satisfy right linkage.
\end{example}

\section{Acknowledgements}

We would like to thank Alex Fink and Michael Joswig for their support and Felipe Rinc{\'o}n, Benjamin Schr{\"o}ter and Kristin Shaw for helpful discussions.
We are grateful to Amanda Cameron for careful reading.
We thank Pavel Galashin, Gaku Liu and Alex Postnikov for communicating the related work~\cite{GalashinNenashevPostnikov2018}.
We would like to thank the referee for helpful comments. 

\bibliographystyle{amsplain}
\bibliography{main}

\end{document}